\documentclass{amsart}
\usepackage[utf8x]{inputenc}
\usepackage{amsmath, amsthm, amssymb}
\usepackage[usenames,dvipsnames,svgnames,table]{xcolor}
\usepackage[margin=1.3in]{geometry}
\usepackage{enumerate}
\usepackage{dsfont}
\usepackage{hyperref}
\usepackage{cleveref}
\usepackage{cite}
\usepackage{mathtools}
\usepackage[normalem]{ulem}
\usepackage{comment}
\usepackage{xfrac}
\usepackage{autonum}
\usepackage{xcolor}
\usepackage{stmaryrd}
\usepackage{cancel}

\usepackage{tikz}
\usepackage{pgfplots}
\pgfplotsset{compat=1.14}
\usetikzlibrary{patterns}
\makeatletter
\pgfdeclarepatternformonly[\Density]{dense dots}
{\pgfpointorigin}
{\pgfpoint{0.5*\Density}{0.5*\Density}}
{\pgfpoint{\Density}{\Density}}%
{
    \pgfsetcolor{\tikz@pattern@color}
    \pgfpathcircle{\pgfpoint{0.25*\Density}{0.25*\Density}}{0.1*\Density}
    \pgfusepath{fill}
}
\makeatother

\Crefname{Assumption}{Assumption}{Assumptions}
\Crefname{Theorem}{Theorem}{Theorems}
\Crefname{Lemma}{Lemma}{Lemmas}
\Crefname{Corollary}{Corollary}{Corollaries}
\Crefname{Proposition}{Proposition}{Propositions}
\Crefname{Theorem}{Theorem}{Theorems}
\Crefname{Conjecture}{Conjecture}{Conjectures}
\Crefname{Remark}{Remark}{Remarks}

\newtheorem{Theorem}{Theorem}[section]
\newtheorem{Proposition}[Theorem]{Proposition}
\newtheorem{Lemma}[Theorem]{Lemma}

\newtheorem{Corollary}[Theorem]{Corollary}

\newcommand{\R}{\mathbb R}

\newcommand{\lbr}{\llbracket}
\newcommand{\rbr}{\rrbracket}

\newcommand{\cI}{\mathcal I}
\newcommand{\cN}{\mathcal N}
\newcommand{\cO}{\mathcal O}
\newcommand{\cP}{\mathcal P}
\newcommand{\cW}{\mathcal W}
\newcommand{\eps}{\varepsilon}

\renewcommand{\H}{\mathbb H^d}

\newcommand{\init}{{\rm in}}

\newcommand{\tand}{\text{ and }}

\DeclareMathOperator{\id}{Id}

\newcommand{\kin}{{\rm kin}}

\renewcommand{\epsilon}{\eps}

\newcommand{\dz}{\, dz}
\newcommand{\dtz}{\, d\tilde z}
\newcommand{\dv}{\,dv}
\newcommand{\dx}{\,dx}
\newcommand{\dt}{\,dt}
\newcommand{\dw}{\,dw}
\newcommand{\dr}{\,dr}
\newcommand{\ds}{\,ds}
\newcommand{\dy}{\,dy}

\DeclareMathOperator{\dist}{dist}

\DeclareMathOperator{\supp}{supp}

\newcommand{\be}{\begin{equation}}
\newcommand{\ee}{\end{equation}}

\def \a {{\alpha}}

\def \R {{\mathbb {R}}}

\def \eps {{\varepsilon}}

\def \phi {{\varphi}}

\def \tilde {\widetilde}

\def \a {{\alpha}}

\renewcommand{\(}{\left(}
\renewcommand{\)}{\right)}

\def \zz {{\rho}}

\DeclareFontFamily{U}{matha}{\hyphenchar\font45}
\DeclareFontShape{U}{matha}{m}{n}{
      <5> <6> <7> <8> <9> <10> gen * matha
      <10.95> matha10 <12> <14.4> <17.28> <20.74> <24.88> matha12
      }{}
\DeclareSymbolFont{matha}{U}{matha}{m}{n}
\DeclareFontSubstitution{U}{matha}{m}{n}

\DeclareFontFamily{U}{mathx}{\hyphenchar\font45}
\DeclareFontShape{U}{mathx}{m}{n}{
      <5> <6> <7> <8> <9> <10>
      <10.95> <12> <14.4> <17.28> <20.74> <24.88>
      mathx10
      }{}
\DeclareSymbolFont{mathx}{U}{mathx}{m}{n}
\DeclareFontSubstitution{U}{mathx}{m}{n}

\DeclareMathDelimiter{\VERT}{0}{matha}{"7E}{mathx}{"17}

\numberwithin{equation}{section}
\newcommand{\domainsfigure}{
\begin{figure}
\begin{tikzpicture}%[rotate=270, xscale=-1]
    \begin{axis}[
	xlabel=$v_1$,
	ylabel=$x_1$,
	thick,
	domain=-3.5:3.5,
	samples=200,
	width=4in,
	height = 3.25in,
	xticklabels={},
	yticklabels={},
	axis x line = bottom,
	axis y line = center,
	axis on top,
	ymin = 0,
	ymax = 3.5,
	rotate=270,
	xscale=-1,
	xlabel style={at={(1.25,0)},anchor=east}
	]

	\addplot[
		pattern = dots,
		pattern color=violet!50,
		draw=none
		] {3.5} \closedcycle;
		
\pgfdeclarepatternformonly{dense horizontal lines}{\pgfpoint{0pt}{0pt}}{\pgfpoint{10pt}{1pt}}{\pgfpoint{10pt}{1pt}}%
        {
            \pgfsetlinewidth{0.5pt}
            \pgfpathmoveto{\pgfpoint{0pt}{1pt}}
            \pgfpathlineto{\pgfpoint{10pt}{1pt}}
            \pgfusepath{stroke}
        }
	\addplot [
		domain=-3.5:-.5,
		samples=200,
		pattern= dense horizontal lines,
		pattern color=black,
		opacity=.5
	] {-.23*x^3} \closedcycle;
		
	\addplot [
		domain=-3.5:.5,
		samples=200,
		fill= white,
		draw=none
	] {(1/7)*ln(exp(-7*(1+x)) + 1)+.4} \closedcycle;
	\addplot [
		domain=-3.5:.5,
		samples=200,
		fill= white,
		fill=red!25,
		draw=none
	] {(1/7)*ln(exp(-7*(1+x)) + 1)+.4} \closedcycle;
	
	\addplot [
		domain=-.8:-.5,	
		samples=200,
		fill = white,
		draw=none
	] {20 * (x+.8)^(1/2)/(1 + 20*(x+.8)^(1/2))} \closedcycle;
	\addplot [
		domain=-.8:-.5,	
		samples=200,
		pattern = crosshatch,
		pattern color=blue!30,
		draw=none
	] {20 * (x+.8)^(1/2)/(1 + 20*(x+.8)^(1/2))} \closedcycle;
	\addplot [
		domain=-.5:3.5,
		samples=200,
		fill = white,
		draw=none
	] {(1/6)*ln(exp(6*(x-1)) + 1)+.92} \closedcycle;	
	\addplot [
		domain=-.5:3.5,
		samples=200,
		pattern = crosshatch,
		pattern color=blue!30,
		draw=none
	] {(1/6)*ln(exp(6*(x-1)) + 1)+.92} \closedcycle;
	\draw[|-|] (.5,.025) -- (.5,.945)
		node[above, midway]{ $O(R^{\sfrac32})$};
	\draw[|-|] (-.78,.25) -- (-.025,.25)
		node[right, midway]{ $O(\sqrt R)$};
	\node[blue,rotate=32.5] at (2.65,2.45){$x_1 = O(Rv_1)$};
	\node[red,rotate=327.5] at (-2.9,2.05){$x_1 = O(R|v_1|)$};
	\node[rotate=352.5] at (-2.2,3){$x_1 = |v_1|^3$};
	
	\node[blue] at (2.5,1) {\huge$\cP_R$};
	\node[violet] at (.5,2.5) {\huge $\cN_R$};
	\node[red] at (-2.5,1) {\huge $\cO_R$};
	\node at (-3,3) {\huge$\cI_R$};
\end{axis}
\end{tikzpicture}

\caption{A cartoon picture of each of the key domains. The Poincar\'e region $\cP_R$ is the blue crosshatched region, the Nash region $\cN_R$ is the violet dotted region, and the outgoing region $\cO_R$ is the red shaded region. The subregion $\cI_R \subset \cN_R$ is the black striped region.  The rough asymptotics of the boundaries separating each region are given as well.}
\label{f.regions}
\end{figure}}

\title[Kinetic Nash and boundary behavior]{A kinetic Nash inequality and precise boundary behavior of the kinetic Fokker-Planck equation}
\author{Christopher Henderson}
\address{Department of Mathematics, University of Arizona, Tucson, AZ 85721}
\email{ckhenderson@math.arizona.edu}
\author{Giacomo Lucertini}
\address{Dipartimento di Matematica, Universit\`a di Bologna, Bologna, Italy, 40126}
\email{giacomo.lucertini3@unibo.it}
\author{Weinan Wang}
\address{Department of Mathematics, University of Oklahoma, Norman, OK, 73019}
\email{ww@ou.edu}

\begin{document}

\begin{abstract}
In this paper, we prove a kinetic Nash type inequality and adapt it to a new functional inequality for functions in a kinetic Sobolev space with absorbing boundary conditions on the half-space.  As an application, we address the boundary behavior of the kinetic Fokker-Planck equations in the half-space.  Our main result is the sharp regularity of the solution at the absorbing boundary and grazing set.
\end{abstract}

\maketitle

\section{Introduction}

\subsection{The equation}
We study the homogeneous kinetic Fokker-Planck equation in the half-space with absorbing boundary conditions:
\be\label{e.kfp}
	\begin{dcases}
		(\partial_t + v\cdot\nabla_x) f = \Delta_v f
			\qquad&\text{ in } \R_+ \times \H \times \R^d,\\
		f(t,x,v) = 0
			\qquad&\text{ on } \R_+ \times \gamma_-,\\
		f(0,\cdot,\cdot) = f_\init
			\qquad&\text{ in } \H \times \R^d,
	\end{dcases}
\ee
where we let $\R_+ = (0,\infty)$, $\R_- = (-\infty,0)$,
\be
	\H = \left\{(x_1,\dots, x_d) \in \R^d : x_1 > 0\right\},
	\quad\text{ and }\quad
	\gamma_{\pm} = \{(x,v) : x_1 = 0, \mp v_1 > 0\}.
\ee
We assume that $f_{\rm in}$ is a nonnegative, measurable function that is an element of a certain weighted $L^1$-space. 
We refer to $\gamma_-$ as the {\em incoming} portion of the boundary and $\gamma_+$ as the {\em outgoing} portion of the boundary.  The sign convention may appear strange above, but we follow the standard notation in the general case: the minus sign corresponds to the negativity of $v\cdot\eta_x$, where $\eta_x$ is the outward pointing unit normal on the physical space boundary.  In our case $\eta_x = (-1,0,\cdots, 0)$.    The set where $x \cdot \eta_x = 0$ is called the ``grazing set.''  In our case this is when $x_1 = 0 = v_1$.

\subsection{Informal discussion of the main results}

Our goal is to understand the precise boundary behavior of~\eqref{e.kfp}.  In particular, we are interested in the sharp regularity on $\gamma_-$.  We note that the interior regularity is quite well-understood; see~\cite{Kolmogorov} for the homogeneous equation and~\cite{HendersonSnelson, imbert2021nonlinear, golse2016, manfredini1997ultraparabolic, HendersonWang, dong2022global, DY_Lp, bramanti2007schauder, difrancesco2006schauder, AbedinTralli, PLP, BiagiBramanti, anceschi2024poincar, AnceschiRebucci, loher2023quantitative, GuerandMouhot,LanconelliPolidoro, anceschi2023fundamental} for more recent results with varying degrees of inhomogeneity.  More generally, we refer to the review~\cite{AnceschiPolidoroReview}.  Let us note that the literature is quite large, so the above is unfortunately only a small sample of related works.  Briefly, though, the major source of difficulty for~\eqref{e.kfp} is the lack of diffusion in $x$.  Instead, one must use ``hypoellipticity'' to import the $v$-regularity (generated by the $\Delta_v$ term on the right hand side) to $(t,x)$-regularity via the transport term $\partial_t + v\cdot\nabla_x$.

To illustrate the boundary regularity, let us briefly introduce a (nontrivial) steady solution to~\eqref{e.kfp}.  As it is convenient to introduce a steady solution to the adjoint problem at the same time, we do so here. These solutions are
\be\label{e.phi}
\begin{dcases}
	v\cdot\nabla_x \phi = \Delta_v \phi
	\quad &\text{ in } \H\times \R^d,\\
	\phi = 0
	\quad &\text{ on } \gamma_-,
\end{dcases}
\quad\tand\quad
\begin{dcases}
	-v\cdot\nabla_x \tilde \phi = \Delta_v \tilde \phi
	\quad &\text{ in } \H\times \R^d,\\
	\tilde \phi = 0
	\quad &\text{ on } \gamma_+.
\end{dcases}
\ee
It is easy to see that,  with an abuse of notation,
\be
	\phi(x,v)
		= \phi(x_1,v_1)
		= \tilde \phi(x_1,-v_1).
\ee
Following \cite[Lemma~2.1]{GroeneboomEtAl}, % and~\cite[Section 7.10.1]{Olver}, 
we have the asymptotics of $\phi$, and, thus, also $\tilde\phi$, given by
\be\label{e.phi_asymp}
	\phi(x,v)
		\approx
			\begin{dcases}
				\tfrac{x_1}{v_1^{\sfrac52}} \exp\left\{- \tfrac{v_1^3}{9x_1}\right\}
					\qquad &\text{ if } 0 \leq x_1 \leq v_1^3,
				\\
				x_1^{\sfrac16}
					\qquad &\text{ if } x_1 \geq |v_1|^3,
				\\
				\sqrt{|v_1|}
					\qquad &\text{ if } 0 \leq x_1 \leq - v_1^3.
			\end{dcases}
\ee
Given this, it is natural to expect that the behavior $f$ is, roughly, exponentially small as $x_1 \to 0$ with $v_1>0$ and $C^{\sfrac16}_xC^{\sfrac12}_v$ as $(x_1,v_1) \to (0,0)$.  This aligns with what is well-understood about kinetic equations: the bottleneck to regularity occurs at the ``grazing set.''

Our goal is to make this precise by both identifying {\em exactly} the behavior conjectured in the previous paragraph and understanding the norms that control $f$ near the boundary.  Our approach is to develop a kinetic boundary Nash inequality that allows for an $L^1_w\to L^2$ estimate, where ``w'' stands for ``weighted.''  By using adjointness, we get then an $L^2\to L^\infty_w$ estimate.  In analogy with the heat equation, one expects
\be\label{e.c042301}
	f(t,x,v)
		\lesssim \frac{\|\tilde \phi f_{\rm in}\|_{L^1}}{t^{2d+\sfrac12}} \phi(x,v),
\ee
where the power of $t$ follows by scaling arguments and the $\tilde \phi$ appears because $\|\tilde \phi f(t)\|_{L^1}$ is a conserved quantity.  To dwell on the last point a moment longer, observe that
\be\label{e.tilde_phi_conserved}
	\begin{split}
		\frac{d}{dt} \int_{\H\times \R^d}  f(t,x,v) \tilde \phi(x,v) \dx\dv
			&= \int_{\H\times \R^d} \left[\left( \Delta_v - v\cdot\nabla_x\right) f\right] \tilde \phi \dx\dv
			\\&
			= \int_{\H\times \R^d}  f \left( \Delta_v + v\cdot\nabla_x\right) \tilde \phi \dx\dv
			= 0.
	\end{split}
\ee
This approach to~\eqref{e.c042301} is outlined in greater detail in \Cref{s.strategy}.  Using standard interior estimates along with~\eqref{e.c042301}, one can easily show that $f$ is $C^{\sfrac12}_{\rm kin} \approx C^{\sfrac13}_tC^{\sfrac16}_xC^{\sfrac12}_v$ up to the boundary and smooth in the interior.

Actually,~\eqref{e.c042301} does {\em not} hold!  Roughly, if $f_{\rm in}$ is supported where $\tilde \phi$ is exponentially small, that is, $v_1 \ll -1$ and $-v_1 \ll x_1 \ll -v_1^3$, the right hand side of~\eqref{e.c042301} will be exponentially small.  On the other hand, $f(1,x,v)$ will remain constant order in $v_1 + \supp(f_{\rm in})$; that is, the set obtained by applying transport to the support of $\supp(f_{\rm in})$. % as it is only acted on by transport in $x$ and diffusion in $v$.  
%Hence, the left hand side of~\eqref{e.c042301} remains constant order.  
This behavior is clearly not consistent with~\eqref{e.c042301}. 
Roughly, this is related to the fact that $\tilde \phi$ ``feels'' infinite time scales while $f$ only ``feels'' the time interval $[0,t]$.  This is important here (and not for the heat equation) because transport does not (locally in $v$) have infinite speed of propagation while diffusion does. 
Regardless,~\eqref{e.c042301} is a good indication of our main result \Cref{t.main} and how the proof proceeds.

In the process of proving our boundary Nash inequality, we develop a whole-space Nash inequality (\Cref{t.ws_Nash}).  This easily yields the sharp time decay estimate
\be\label{e.c061910}
	f(t,x,v)
		\lesssim \frac{\|f_{\rm in}\|_{L^1}}{t^{2d}},
\ee
for solutions of~\eqref{e.kfp} posed on $\R_+ \times \R^d \times \R^d$.  Actually, one can easily include a uniformly elliptic (rough) diffusion matrix in~\eqref{e.kfp} in the arguments deriving~\eqref{e.c061910}.  This is the content of \Cref{c.ws_time_decay}.  It is interesting to note that estimates of this form, suitably weighted, have been used in the parabolic setting to obtain Harnack inequalities and regularity in the classic work of Fabes and Stroock~\cite{FabesStroock}.  It is possible that this could provide a new method to understand estimates of the fundamental solution.  See~\cite{LanconelliPascucci} for a related approach based on a kinetic Sobolev inequality and Moser's iteration.

%
%We explain this in greater detail in the sequel.  Along the way, we establish a general Nash-type inequality relating the $L^1$, $L^2$, and $H^1_{\rm kin}$-norms of a function and, as it comes `for free' from this Nash inequality, we deduce the whole space decay estimate
%\be
%	f(t,x,v)
%		\lesssim \frac{\|f_{\rm in}\|_{L^1}}{t^{2d}},
%\ee
%for solutions of~\eqref{e.kfp} posed on $\R_+ \times \R^d \times \R^d$.
%

\subsection{Precise statements of main results: boundary behavior on the half-space}

\begin{Theorem}\label{t.main}
	Suppose that $f$ solves~\eqref{e.kfp}.  There is a constant $\alpha > 0$ and a nonnegative smooth function $\mu$ bounded by 1, satisfying
	\be
		\mu(t, x,v)
			=\mu(t,x_1,v_1)
			\approx \begin{cases}
					0
						\qquad &\text{ if } v_1 <  \alpha \sqrt t
							\text{ or } x_1 \geq \frac{|v_1|^3}{\alpha},
					\\
					1
						\qquad &\text{ if } \alpha t v_1 \leq x_1 \leq v_1^3 \a,
					\\
					e^{-\frac{\alpha t v_1}{x_1}}
						\qquad &\text{ if } v_1 \geq \alpha \sqrt t
							\text{ and } x_1 \leq  \alpha t v_1,
				\end{cases}
	\ee
	such that $f$ may be decomposed as
	\be
		f(t,x,v) = \phi(x,v) h_1(t,x,v) + t^{\sfrac14} \mu(x,v) h_2(x,v),
	\ee
	where, for $i=1,2$,
	\be\label{e.c061908}
		\|h_i(t,\cdot,\cdot)\|_{L^\infty(\H\times \R^d)}
			\lesssim \frac{1}{t^{2d+\sfrac12}} \Big( \int f_\init \tilde \phi \dx\dv
					+ t^{\sfrac14} \int f_\init \tilde \mu \dx\dv\Big),
	\ee
	where $\tilde \mu (t,x_1,v_1) = \mu(t,x_1,-v_1)$ (see \Cref{s.notation}).
%	\be\label{e.c061908}
%		\inf_h \Big( \Big\|\frac{f-h}{\phi}\Big\|_{L^\infty(\H\times \R^d)}
%			+ \Big\|\frac{h}{t^{\sfrac14} \mu}\Big\|_{L^\infty(\H\times \R^d)}\Big)
%				\lesssim \frac{1}{t^{2d+\sfrac12}} \Big( \int f_\init \tilde \phi \dx\dv
%					+ t^{\sfrac14} \int f_\init \tilde \mu \dx\dv\Big).
%	\ee	
\end{Theorem}
\Cref{t.main} is quite a bit to digest, so let us discuss it briefly.  First, $\mu$ is defined in \Cref{l.isolated_estimate} (note: $\tilde\mu(t,x,v) = \mu(t,x,-v)$ in \Cref{l.isolated_estimate}). 

Second, let us consider the simple case where $f_\init$ is compactly supported. Then, for $t$ sufficiently large, the right hand side of~\eqref{e.c061908} reduces to
\be
	\frac{1}{t^{2d+\sfrac12}} \int f_\init \tilde \phi \dx\dv.
\ee
Let us also only consider here the case $v_1 \leq \alpha \sqrt t$.

In this case, we see the following behavior near the ``grazing set'' $x_1 = v_1 = 0$: if $0 < x_1, |v_1| \ll 1$,
\be
	f(t,x,v)
		= \phi(x,v) h_1(x,v)
		\lesssim \frac{\phi(x,v)}{t^{2d+\sfrac12}}.
\ee
Using~\eqref{e.phi_asymp}, we see precisely the $C^{\sfrac16}_x C^{\sfrac12}_v$-regularity at $(t,0,0)$.

Next, consider the behavior near $\gamma_-$: fix any $v_1 \in (0, \alpha \sqrt t)$ and take $0 < x_1 \ll 1$.  Similarly to the above, we find
\be
	f(t,x,v)
		\lesssim \frac{\phi(x,v)}{t^{2d+\sfrac12}}
		\approx \frac{x_1}{ v_1^{\sfrac52} t^{2d+\sfrac12}} e^{-\frac{v_1^3}{9x_1}}.
\ee
In other words, we recover a precise form of the super-polynomial decay observed by Silvestre in~\cite{SilvestreBoundary}.  It should be noted that Silvestre considers a much more irregular model than~\eqref{e.kfp}.

The case when $v_1 \geq \alpha \sqrt t$ is essentially the same, although with the addition of an exponentially decaying (in $\sfrac{v_1}{x_1}$) term due to $\mu$.  Thus, just as in the previous case, we see ``fast'' decay in $\sfrac{v_1}{x_1}$.

As we mentioned above, the bottleneck to regularity up to the boundary is precisely in understanding the decay of $f$ as $x_1 \to 0$.  As such, it is straightforward to use interior regularity estimates, suitably scaled, to deduce that
\be
	f \in C^{\sfrac12}_\kin \approx C^{\sfrac12}_t C^{\sfrac16}_x C^{\sfrac12}_v
\ee
from \Cref{t.main}; see~\cite{HJV,HJJ} for one approach to this.  We omit the details.  Since it is not the main focus of this work, we also do not clarify precisely the spaces $C^\alpha_\kin$ beyond the rough statement above.

Finally, let us discuss the meaning and necessity of the $\mu$ and $\tilde \mu$ terms.  As referenced in the discussion of~\eqref{e.c042301}, they arise due to the ``isolated'' region
\be\label{e.c062008}
	I_t = \{ (x,v) : v_1 \leq - O(\sqrt t), O(t) v_1 \leq x_1 \leq v_1^2\}.
\ee
This set contains particles that are too far from the outgoing boundary $\gamma_+$ to travel there by transport in time $t$ and are too far from the incoming boundary $\gamma_-$ to have made it there following transport for time $t$ and then making ``jump'' in velocity of size $O(\sqrt t)$.  The ``allowed'' jump size is determined by scaling, although it comes up in more concrete ways in our arguments.

Given this isolation, one expects the $L^1$-norm of $f$ on $I_t$ to be roughly constant for times $[0,t]$.  From a microscopic point of view, this says that the density of particles in $I_t$ is roughly constant.  Intuitively, particles can leave $I_t$ in two ways.  First, a particle can make a velocity jump, leaving $I_t$ through the top.  Here $\tilde \phi$ is ``large'' and we can control this quantity with a term of the form
\be
	\frac{1}{t^{\sfrac14}} \int f \tilde \phi \dx \dv
		= \frac{1}{t^{\sfrac14}} \int f_\init \tilde \phi \dx \dv
\ee
(recall~\eqref{e.tilde_phi_conserved}).  Let us note that the time scaling is not obvious at this point.  Second, a particle can follow transport and leave $I_t$ through the {\em left} (because $v_1 < 0$).  This is accounted for by the exponential part of $\tilde\mu$, which is the appropriate density for these dynamics.  Indeed,
\be
	(\partial_t - v\cdot\nabla_x - \Delta_v) e^{\frac{\alpha t v_1}{x_1}}
		\leq 0
\ee
for $x_1 \leq - O(t) v_1$ and $v_1 \leq - O(\sqrt t)$.

\subsubsection{Previous results}

The closest works to ours are those of Hwang, Jang, and Vel\'azquez~\cite{HJV} and Hwang, Jang, and Jung~\cite{HJJ}; see also~\cite{HwangKim}.  In these remarkable works, the authors prove many results, the most relevant to the current work being the $C^{\sfrac\alpha3}_xC^\alpha_v$-regularity of $f$ for any $\alpha < \sfrac12$ given $f_\init \in L^1\cap L^\infty$.  They prove that the decay rate at the boundary controls the regularity.  To understand the decay rate, they construct highly nontrivial supersolutions by a clever change-of-variables and a careful patching of special functions.  
Our approach is quite different than their comparison principle based one, and one advantage is that we are able to identify the precise regularity, time decay, and controlling quantities (the $L^1_{\tilde \phi}$ and $L^1_{\tilde \mu}$-norms of $f_\init$) of the boundary behavior.  %On the other hand, by performing a local boundary flattening, they are able to handle general $C^3$ domains.

A more general approach is given by the De Giorgi methods of Silvestre~\cite{SilvestreBoundary} and Zhu~\cite{zhu2022regularity}.  Allowing rough coefficients in~\eqref{e.kfp}, these works obtain $C^\alpha_\kin$ estimates of $f$, where $\alpha$ depends on the bounds of the coefficients.  Silvestre also observed that, as $(x,v) \to (0,v_+)$ with $v_+ > 0$, $f(t,x,v) \lesssim x^p$ for any $p>0$.  As discussed above, we obtain a precise version of this.  We also mention the recent preprint~\cite{hou2024boundedness}.

Let us finally note that hypocoercivity is another approach to overcoming the lack of diffusion in $x$ for kinetic equations.  We point out Villani's classic memoir~\cite{hypocoercivity} for a discussion of this topic; however, this area remains quite active.  See, for example,~\cite{carrapatoso2024kinetic, BDMMS, BouinDolbeaultLafleche}.  That approach is quite different from our own.

\subsubsection{Generalizations}

It is clear that, for a general convex domain $\Omega_x$, our results immediately give, via the comparison principle, the upper bound in \Cref{t.main} when~\eqref{e.kfp} is posed on $\R_+ \times \Omega_x \times \R^d_v$.  One need only rotate and translate $\{0\}\times \R^{d-1}$ to be a supporting hyperplane of $\partial \Omega_x$.

A more interesting question is how to generalize the results to the case of a general nonconvex domain $\Omega$ or the case with nonconstant coefficients
\be\label{e.general}
	(\partial_t + v\cdot\nabla_x ) f = \nabla_v( a \nabla_v f) + \text{(lower order terms)}.
\ee
Let us focus on the latter as the former is, in some sense, a subcase of the after applying a suitable boundary flattening change of coordinates.

If $a \equiv \id$, then our results above are immediately applicable to obtain $x_1^{\sfrac16}$ and $v_1^{\sfrac12}$ decay near $x_1 = 0 =v_1$.  The only difference is that the lower order terms may cause norm growth, so that the $t^{2d+1}$ term in the numerator of \Cref{t.main} may be changed.

When $a \neq \id$ and $a$ is sufficiently smooth, a change of variables and a rescaling takes $a$ to the identity plus a small perturbation, locally.  This is a typical technique in the proof of Schauder estimates (see, e.g.,~\cite[Section~2.2]{HendersonWang}).  In principle, one should be able to use this to recover the $x_1^{\sfrac16}$ and $v_1^{\sfrac12}$ decay estimates in \Cref{t.main}.  %This is the subcritical case.

When $a$ is ``rough,'' one does not expect the $x_1^{\sfrac16}$ and $v_1^{\sfrac12}$ decay to hold by analogy with (divergence form) elliptic equations.  In this case, the results of Silvestre~\cite{SilvestreBoundary} and Zhu~\cite{zhu2022regularity} are likely the best one can hope for: $C^\alpha_\kin$-regularity up to the boundary with $\alpha$ depending on the ellipticity bounds of $a$.  %This is the critical case.

Let us point out that an advantage to our approach is the boundary behavior of {\em generic} solutions reduces to understanding the boundary behavior of a {\em single} solution to each of the equation~\eqref{e.kfp} and the adjoint equation~\eqref{e.akfp}.  Here, we use the steady solution; however, significantly less is actually required.  Indeed, we only use mild control of the asymptotic growth of $\tilde \phi$ in certain regimes (e.g., \Cref{l.cN_R}) and that the growth of
\be
	\int f(t,x,v) \tilde \phi(x,v) \dx \dv
\ee
is controlled in time.  Thus, in the general case~\eqref{e.general}, we need only find $g$ with the appropriate boundary behavior and asymptotic growth in $\cN_R$ such that
\be
	\int f(t,x,v) g(t,x,v) \dx\dv
\ee
(at most) grows in a controlled way.  This last requirement is true of any function $g$ such that
\be
	(\partial_t + v\cdot\nabla_x + \nabla_v\cdot a\nabla_v) g
		\lesssim g.
\ee

\subsection{Precise statements of main results: Nash inequalities and the whole space case}

As we discuss in \Cref{s.strategy}, we obtain the main functional inequality (\Cref{l.functional}) for \Cref{t.main} by interpolating between boundary Poincar\'e-type inequalities and a localized Nash inequality.  The localized Nash inequality may be of independent interest, so we state it here.  Let us note that the kinetic notation $\delta_\cdot$, $\cdot^{-1}$, and $H^1_\kin$ are defined in \Cref{s.kinetic}.

\begin{Theorem}\label{t.ws_Nash}
	Fix $s_0 > 0$ and sets $\Omega_1, \Omega_2 \subset \R_+ \times \R^{2d}$ such that there is a bounded open set $B$ with
	\be
		\Omega_1 \circ (\delta_s B)^{-1} \subset \Omega_2
			\qquad\text{ for all } s \in [0,s_0].
	\ee
	Then, for any $g \in H^1_\kin$ and $s\in (0,s_0]$, we have
	\be
		\|g\|_{L^2(\Omega_1)}^2
			\lesssim s \lbr g\rbr_{H^1_\kin(\Omega_2)}\|g\|_{L^2(\Omega_2)}
				+ \frac{1}{s^{4d+2}} \|g\|_{L^1(\Omega_2)}^2.
	\ee
	The implied constant depends only on the choice of $B$ and the dimension.
\end{Theorem}

With this in hand, we can immediately deduce a simple time-decay estimate for the whole-space kinetic Fokker-Planck equation.  This estimate is not new; one can derive it from existing results on fundamental solutions; see, e.g.,~\cite{anceschi2023fundamental, LanconelliPascucci}, although these proofs are quite different from our own.  We only include it here because it is essentially immediate from \Cref{t.ws_Nash}.  It is not our main interest in this paper.

\begin{Corollary}\label{c.ws_time_decay}
	Suppose that $a$ is a symmetric, uniformly elliptic matrix:
	\be
		|\xi|^2 \lesssim \xi \cdot a(t,x,v) \xi
			\qquad\text{ for all } (t,x,v) \in \R_+ \times \R^{2d}
			\text{ and } \xi \in \R^d.
	\ee
	If $f$ is a nonnegative solution to
	\be\label{e.ws_kfp}
		\begin{cases}
			(\partial_t + v\cdot\nabla_x) f = \nabla_v\cdot\left( a \nabla_v f\right)
				\qquad& \text{ in } \R_+\times \R^{2d},\\
			f = f_\init
				\qquad &\text{ on } \{0\} \times \R^{2d},
		\end{cases}
	\ee
	then
	\be
		f(t,x,v) \lesssim \frac{1}{t^{2d}} \int f_\init \dx \dv.
	\ee
\end{Corollary}

If one includes lower terms such as $b \cdot \nabla_v f + c f$ in~\eqref{e.ws_kfp}, the bounds above will hold with (possibly) an additional exponentially growing in $t$ factor depending only on $\|c\|_\infty$ and $\|b\|_\infty$.

Finally, we note that the well-posedness of~\eqref{e.kfp} and~\eqref{e.ws_kfp} with merely weighted $L^1_w$ initial data follows simply using ideas in~\cite{HJV,zhu2022regularity} and standard approximation schemes.  By the established regularity theory, solutions will be classical in the interior (and up to the boundary in $x$) and continuous in time up to $t=0$ in $L^1_w$. 
As such, we omit further discussion of this.

\subsection{Organization of the paper}

To aid the reader, we give a discussion of the general strategy of the proof in the parabolic setting in \Cref{s.strategy}.  It is here that we also give an indication of the main difficulties in the paper. 

The main functional analysis and group theory setup that is appropriate for kinetic equations is given in~\Cref{s.kinetic}.  

The proof of \Cref{t.main} occurs in \Cref{s.main,s.functional}.  The former contains the proof of \Cref{t.main} subject to a few inequalities that are stated there.  The main inequality  stated in \Cref{s.main} (\Cref{l.functional}) relies on a decomposition of $\H\times \R^d$ into a ``Nash'' region $\cN_R$ and two ``Poincar\'e'' regions $\cP_R$ and $\cO_R$.  See \Cref{f.regions}.  This main inequality, proved in \Cref{s.functional}, follows by establishing a localized Nash inequality in $\cN_R$ and Poincar\'e-type inequalities in $\cP_R$ and $\cO_R$.  These proofs are also contained in \Cref{s.functional}.

The construction of $\mu$ occurs in \Cref{s.zeta}, and several technical lemmas are proved in \Cref{s.technical}.

Finally, the whole space case is briefly considered in \Cref{s.ws_time_decay}.

\subsection{Notation}\label{s.notation}

We use $z$ to denote a generic point $(t,x,v)$.  When $z$ is decorated with notation, the coordinates inherit that decoration; e.g., $z' = (t',x',v')$.

We write $A\lesssim B$ is $A \leq CB$ for a constant $C$ depending only on dimension.  We write $A\approx B$ if $A\lesssim B$ and $B \lesssim A$.

In order to clearly define when we use the dynamics associated to~\eqref{e.kfp}, we reserve $f$ for its solutions and use $g$ (or other letters) for any generic element of $H^1_\kin$.

Whenever the domain of integration is not specified, it is assumed to be in $\H\times \R^d$ if it is an integral with respect to $dv dx$, $\H$ if it is an integral with respect to $dx$, or $\R^d$ if it is an integral with respect to $dv$.

We write $v = (v_1, \overline v)$, where $\overline v \in \R^{d-1}$. Similarly, $x = (x_1, \overline x)$.  We use the tilde to denote reflection in $v$:
\be\label{e.tilde}
	\tilde f(t,x,v) = f(t,x,-v).
\ee
This is defined similarly for functions that depend only on $(x,v)$ or only on $v$. 
We use the star to denote taking the adjoint of an operator; that is $A^*$ is the adjoint of an operator $A$.  Note that this is some overlap here because the adjoint equation of~\eqref{e.kfp} is
\be\label{e.akfp}
	\begin{dcases}
		(\partial_t - v\cdot\nabla_x) \tilde f = \Delta_v \tilde f
			\qquad&\text{ in } \R_+ \times \H \times \R^d,\\
		\tilde f(t,0,v) = 0
			\qquad&\text{ on } \R_+ \times \gamma_+,\\
		\tilde f(0,\cdot,\cdot) = \tilde f_\init
			\qquad &\text{ in } \H\times \R^d
	\end{dcases}
\ee
whose solution is $\tilde f$ if $f$ solves~\eqref{e.kfp}.
%
%whose steady solution is $\phi$, is~\eqref{e.akfp -- WHERE TO PUT THIS?} has the steady solution $\tilde \phi$.  More generally, if $f$ is a solution to~\eqref{e.kfp}, then $\tilde f$ is a solution to

We sometimes use $Y$ as shorthand for the transport operator:
\be
	Y = \partial_t + v\cdot\nabla_x.
\ee
While there are some downsides to this notation -- it is opaque and it suppresses the dependence on $v$ -- it simplifies many expressions significantly and it follows a standard convention.

\subsection{Acknowledgements}

CH thanks Juhi Jang for helpful discussions of the results in~\cite{HJV,HJJ}.  CH was supported by NSF grants DMS-2204615 and DMS-2337666. 
GL  was partially supported by INdAM-GNAMPA Project ``Stochastic mean field models: analysis and applications.''
WW was partially supported by an AMS-Simons travel grant.

\section{The strategy of the proof}\label{s.strategy}

\subsection{Boundary behavior for the heat equation}

Let us recall a simple approach to understanding the boundary behavior for the heat equation in one dimension:
\be\label{e.heat}
	\begin{cases}
		h_t =  h_{xx}
			\qquad &\text{ in } \R_+ \times \R_+,\\
		h(t,0) = 0
			\qquad &\text{ for all } t>0,\\
		h(0,x) = h_\init(x)
			\qquad &\text{ for all } x > 0.
	\end{cases}
\ee
This will give the basic outline of our proof for the kinetic Fokker-Planck equation~\eqref{e.kfp}.

We observe that the equation above is formally self-adjoint and $x$ is a steady solution to it; hence,
\be\label{e.c062006}
	\frac{d}{dt} \int_0^\infty x h \dx
		=  \int_0^\infty x h_{xx} \dx
		= 0.
\ee
Next, we notice the energy equality
\be\label{e.energy3}
	\frac{d}{dt} \int_0^\infty h^2 \dx
		= - \int_0^\infty |h_x|^2 \dx.
\ee
In the whole space, it suffices to use the Nash inequality,
\be
	\Big( \int g^2 \dx\Big)^3
		\lesssim \Big( \int |g_x|^2 \dx\Big)
			\Big( \int g \dx\Big)^4
			\qquad\text{ for all } g \geq 0,
\ee
to control the right hand side of~\eqref{e.energy3}.  However, we need to use the added information in~\eqref{e.c062006}.

To this end, we fix an arbitrary $R>0$, apply the Poincar\'e inequality on $(0,R)$ and the Nash inequality on $(R,\infty)$: for any $g$,
\be\label{e.close_far}
	\int_0^\infty g^2 \dx
		\lesssim R^2 \int_0^R |g_x|^2 \dx
			+ \Big( \int_R^\infty |g_x|^2 \dx\Big)^{\sfrac13}
			\Big( \int_R^\infty g \dx\Big)^{\sfrac43}.
\ee
Here we are assuming that the Nash inequality can be localized.  The usual proof using the Fourier transform does not allow this, but it is not difficult to develop a different proof that does.  Then we add an $\sfrac{x}{R}$ factor to the $L^1$-term to obtain
\be\label{e.smuggle_steady}
	\int_0^\infty g^2 \dx
		\lesssim R^2 \int_0^R |g_x|^2 \dx
			+ \frac{1}{R^{\sfrac43}}\Big( \int_R^\infty |g_x|^2 \dx\Big)^{\sfrac13}
			\Big( \int_R^\infty x g \dx\Big)^{\sfrac43}.
\ee
Optimizing in $R$ yields
\be\label{e.he_Nash}
	\int_0^\infty g^2 \dx
		\lesssim \Big( \int_0^\infty |g_x|^2 \dx\Big)^{\sfrac35}
			\Big( \int_0^\infty x g \dx\Big)^{\sfrac45}.
\ee
Applying~\eqref{e.he_Nash} to $h$ and folding it into~\eqref{e.energy3}, we deduce
\be\label{e.diff_ineq}
	\frac{d}{dt} \int_0^\infty h^2 \dx
		\lesssim - \frac{ \Big( \int h^2 \dx\Big)^{\sfrac53}}{\Big(\int x h \dx \Big)^{\sfrac43}}
		= - \frac{ \Big( \int h^2 \dx\Big)^{\sfrac53}}{\Big(\int x h_{\rm in} \dx \Big)^{\sfrac43}},
\ee
where the equality is due to~\eqref{e.c062006}. 
Solving this differential inequality gives us the desired $L^1_x \to L^2$ bound:
\be
	\Big(\int_0^\infty h(t,x)^2 \dx\Big)^{\sfrac12}
		\lesssim \frac{1}{t^{\sfrac34}} \int_0^\infty xh_\init \dx.
\ee
Letting $S_t: L^1_x \to L^2$ be the solution operator to~\eqref{e.heat}, this translates to
\be
	\|S_t\|_{L^1_x\to L^2} \lesssim \frac{1}{t^{\sfrac34}}.
\ee
On the other hand, the adjoint operator $S^*_t : L^2 \to L^\infty_{\sfrac{1}{x}}$ is also a solution operator to~\eqref{e.heat} because~\eqref{e.heat} is formally self-adjoint and must satisfy
\be
	\|S_t^*\|_{L^2 \to L^\infty_{\sfrac1x}}
		= \|S_t\|_{L^1_x\to L^2} \lesssim \frac{1}{t^{\sfrac34}}.
\ee
Hence, we have
\be
	\|h(t)\|_{L^\infty_{\sfrac1x}}
		= \|S^*_{\sfrac{t}{2}}S_{\sfrac{t}{2}}h_\init\|_{L^\infty_{\sfrac1x}}
		\lesssim \frac{1}{t^{\sfrac34}} \|S_{\sfrac{t}{2}}h_\init\|_{L^2}
		\lesssim \frac{1}{t^{\sfrac34}} \frac{1}{t^{\sfrac34}} \|h_\init\|_{L^1_x}.
\ee
In other words,
\be
	h(t,x) \lesssim \frac{x}{t^{\sfrac32}} \int y h_\init \dy,
\ee
which provides the desired (sharp) boundary regularity.

\subsection{Basic ideas in the kinetic setting}

\subsubsection*{The energy equality and the \texorpdfstring{$H^1_\kin$-norm}{kinetic Sobolev norm}}
Let us point out the basic changes that must occur to put the above plan into action.  First, we already see a difference in the energy equality for~\eqref{e.kfp}:
\be\label{e.energy}
	\frac12\frac{d}{dt} \int f^2 \dx\dv
		+ \int |\nabla_v f|^2 \dx\dv
		+ \int_{\gamma_+} |v_1| f d\overline x \dv
		= 0,
\ee
where $x = (x_1, \overline x)$.  One might be tempted to drop the boundary term above since it has a ``good'' sign; however, we see below that this is not possible.  

Using the definition of the $H^1_\kin$-norm in~\eqref{e.H1_kin_norm},  we immediately obtain, from~\eqref{e.kfp},
\be\label{e.H1_kin_f}
	\lbr f\rbr_{H^1_\kin([T_1,T_2]\times \H\times \R^d)} \approx \|\nabla_v f\|_{L^2([T_1,T_2]\times \H\times \R^d)}.
\ee
In this sense, we immediately obtain bounds on the $H^1_\kin$-norm of $f$ by integrating~\eqref{e.energy} in time.

At this point, we notice our first roadblock to the strategy above: the $H^1_\kin$-norm involves a time integral, meaning that any inequality following from a Nash-type inequality will involve time integrals.  Thus, no differential inequality, such as~\eqref{e.diff_ineq} is possible.  This, however, is  not too difficult to overcome -- it essentially amounts to using the integral form of Gr\"onwall's inequality instead of the differential form.  

\subsubsection*{The Poincar\'e bound}
Next, after determining the appropriate notion of distance, we may start to follow the decomposition in~\eqref{e.close_far}.  First, we can define the set $\cP_R$ of points $(x,v)$ within distance $R$ to the boundary $\gamma_-$ on which we have zero boundary data.  See \Cref{f.regions}. 
\domainsfigure
This requires some technical care, but follows a general method of proving the Poincar\'e inequality by integrating $Yf$ and $\nabla_v f$ along a path starting on $\gamma_-$.  Here we are able to follow the ideas of~\cite{AAMN} to obtain an inequality {\em like}
\be\label{e.c062007}
	\|f\|_{L^2([T, T+R]\times \cP_R)}
		\lesssim R\|f\|_{H^1_{\rm kin}([T, T + O(R)]\times \cP_{2R})}.
\ee
See \Cref{p.poincare} for the actual inequality.

\subsubsection*{The outgoing region}
Next, by analogy with the heat equation, one might hope to have a Nash inequality on $\cP_R^c$ and follow the step~\eqref{e.smuggle_steady} in which the steady solution is brought into the integral up to an $R$ factor.  For this, we would need $\tilde \phi \gtrsim R^p$, for some $p$, on $\cP_R^c$.  In view of~\eqref{e.phi_asymp}, $\tilde \phi$ is exponentially small when $x_1 \ll - v_1^3$.  Hence, this is not immediately possible.

This leads us to the observation that many of the particles in $\cP_R^c$ leave the domain through $\gamma_+$.  Defining $\cO_R$ to be these outgoing particles (over a time interval of size $O(R)$), we can argue as in the Poincar\'e case to obtain a similar inequality to~\eqref{e.c062007} that includes the boundary term from~\eqref{e.energy}.

It is easy to see $\cO_R$ is approximately those particles such that $v_1 \leq - O(\sqrt R)$ and %, when $v_1$ is ``positive enough,'' 
$x_1 \leq - O(R) v_1$.  The latter reflects that particles can be taken to the boundary by pure transport over time $O(R)$.

\subsubsection*{The Nash region}
At this point, we have no choice but to take the Nash inequality on the set $\cN_R$ of points greater than distance $R$ from $\gamma_\pm$.  See \Cref{t.ws_Nash} and \Cref{p.Nash}.  The issue is in connecting the $L^1$-norm that appears there with $\tilde \phi$.  When $x_1 \geq -O(R) v_1^3$, we have $\tilde \phi\gtrsim R^{\sfrac14}$, and we can argue exactly as in~\eqref{e.smuggle_steady}.  This, however, is not the entirety of $\cN_R$.

We prove the Nash inequality via an interpolation argument.  It proceeds by suitably smoothing $g$ to obtain $g_\eps$, writing
\be
	\|g\|_{L^2}
		\leq \|g_\eps\|_{L^2} +  \|g_\eps - g\|_{L^2},
\ee
bounding the first term by the $L^1$-norm of $g$ via a kinetic Young's inequality, and then bounding the second term by the $H^1_\kin$-norm of $g$.  This second term requires some technical care due to the $H^{-1}_v$-$H^1_v$ pairing in the $H^1_\kin$-norm (see \eqref{e.H1_kin_norm}).  The result \Cref{p.Nash} follows by varying $\eps$.

\subsubsection*{The isolated region}
 This leaves the isolated region $\cI_R\subset \cN_R$ of points $- O(R) v_1 \leq x_1 \leq - v_1^3$, on which $\tilde \phi$ is small.  This is the region where, phenomenologically, the behavior of~\eqref{e.kfp} is most different from~\eqref{e.heat}.  In the other regions, the computations, while technically more complicated, bore some resemblance towards their analogues for the heat equation.

This region and its role was discussed around~\eqref{e.c042301} and~\eqref{e.c062008}.  There it is pointed out that no inequality is possible purely using $\tilde \phi$.   
As mentioned there, we overcome this by the construction of a weight $\mu_R$ that encapsulates the movement of particles into and out of $\cI_R$.  We summarize by noting that we get an inequality like
\be
	\|f\|_{L^1([T,T+R]\times \cI_R)}
		\lesssim R^{\sfrac34} \int f_\init \tilde \phi \dx\dv 
			+ R^{\sfrac34} \int f_\init \tilde \mu_R \dx\dv.
\ee
See \Cref{l.isolated_estimate} and~\eqref{e.c062009}.

\section{Kinetic functional analysis}\label{s.kinetic}

\subsection{\texorpdfstring{The functional space $H^1_\kin$}{The kinetic Sobolev space}}
Let us define the space
\be
	H^1_{\kin,0}((T_1,T_2)\times \Omega\times \R^d)
		= \{f \in H^1_\kin((T_1,T_2) \times \Omega) : f(t,x,v) = 0 \ \text{ if } (x,v)\in \partial_\kin \Omega\}.
\ee
where
\be
	\partial_\kin \Omega
		= \left\{(x,v) : x \in \partial \Omega, v \cdot \eta(x) < 0\right\},
\ee
and $\eta(x)$ is the outward pointing normal vector to $\partial \Omega$.  We define the semi-norm on this by
\be\label{e.H1_kin_norm}
	\lbr f\rbr_{H^1_\kin}
		= \|\nabla_v f\|_{L^2}
			+ \sup_{h \in H^1_\kin, \|\nabla_v h\|_{L^2} = 1} \int_{T_1}^{T_2} \int_{\Omega \times \R^d} (Yf)(t,x,v) h(t,x,v) \dv\dx \dt.
\ee
In some sense, the last integral should really be understood as an $H^{-1}_v$-$H^1_v$ pairing in the $v$-variable that is equal to the integral if $u$ and $g$ are sufficiently smooth.  We abuse notation, however, and simply write the integral.  This is justified due to the density of smooth functions; see discussion in~\cite{AAMN}.  A norm on $H^1_{\kin,0}$ is obtained by including the $L^2$-norm as well.  One can then construct $H^1_{\kin,0}$ as the closure of $C^\infty_c$ functions under this norm.

Let us note that there is not accepted convention on the ``correct'' kinetic Sobolev space.  There are several approaches to kinetic Besov and Sobolev spaces, e.g.~\cite{AAMN, GarofaloTralli, PascucciPesce}.  We use the one proposed by Albritton, Armstrong, Mourrat, and Novack in~\cite{AAMN} as it appears to pair well with the equation~\eqref{e.kfp}.  Indeed, using~\eqref{e.H1_kin_norm}, one immediately obtains an $H^1_\kin$-bound from the energy equality~\eqref{e.energy} (see  the discussion below~\eqref{e.energy}).

\subsection{The Lie group structure, kinetic distance, and kinetic convolution}

To aid the reader, let us review standard facts on the scaling and Lie group structure relevant to kinetic Fokker-Planck equations.  This simplifies many arguments notationally and technically.

The equation~\eqref{e.kfp} has a 2-3-1 scaling law; that is, it is invariant under dilations
\be\label{e.dilation}
	\delta_r z = (r^2 t, r^3 x, r v).
\ee
Given $z, z'$, we define
\be\label{e.Galilean}
	z\circ z' = (t + t', x + x' + t' v, v + v').
\ee
Roughly, this reflects the structure of~\eqref{e.kfp} that allows mass to move by diffusion in $v$ and by transport in $(t,x)$.  Indeed, if a unit of mass is at $(x,v)$ and we move forward in time by $t'$ units, our mass shifts to $x \mapsto x + t'v$.  
In fact, one sees that, for a fixed $z_0$, $\tilde f(z) = f(z_0 \circ z)$ solves the first equation in~\eqref{e.kfp}.  This is sometimes referred to as the Galilean invariance of~\eqref{e.kfp} and is at the heart of why $\circ$ is the appropriate nothing of translation.

Clearly this action is invertible with
\be
	z^{-1} = (-t, -x + tv, -v),
\ee
whence
\be
	\begin{split}
		&z^{-1} \circ \tilde z
			 = (\tilde t - t,\tilde  x - x - (\tilde t- t) v, \tilde v - v)
		\quad \text{ and}
		\\&
		z \circ \tilde z^{-1}
			= (t-\tilde t, x -\tilde x - \tilde t (v-\tilde v), v - \tilde v).
	\end{split}
\ee
Note the group action is non-commutative but associative.

Given two sets $A, B \subset \R^{2d + 1}$, we can analogously define the Lie action between them:
\be\label{e.circ_sets}
	A\circ B = \{a\circ b : a \in A, b \in B\}.
\ee
and we also define the set of inverses
\be
	B^{-1} = \{b^{-1} : b \in B\}.
\ee
This will play a role in understanding how $A$ and $A_\eps$ relate to each other integrals of the form
\be
	\int_{A_\eps} u_\eps \dz
		\approx \int_A u \dz,
\ee
where $u_\eps$ is defined via convolution of $u$ with a compactly supported mollifier (see \Cref{l.Youngs}).

There are several norms and distances that one may choose.   Here, we follow~\cite{imbert2020smooth} and use
\be\label{e.norm}
	d_\kin(z',z) = \VERT z^{-1} \circ z'\VERT
		\quad\tand\quad
	\VERT z\VERT
		= \min_{w\in \R^d} \left[ \max\left\{ |t|^{\sfrac12}, |x - t w|^{\sfrac13}, |v-w|, |w|\right\}\right].
\ee
We point out that this norm respects the 2-3-1 scaling and Galilean invariance of the equation~\eqref{e.kfp}.  Indeed, 
\be
	d_\kin(\delta_r z_1, \delta_r z_2)
		= r d_\kin(z_1, z_2)
	\quad\tand\quad
	d_\kin(z \circ z_1, z\circ z_2)
		= d_\kin(z_1, z_2).
\ee

An advantage to this choice of distance and norm is that
\be
	\VERT z \circ z'\VERT
		\leq \VERT z\VERT + \VERT z'\VERT,
\ee
so that the triangle inequality for $d_\kin$ holds:
\be\label{e.defQ}
	\begin{split}
		d_\kin(z_1, z_3)
			&= \VERT z_3^{-1} \circ z_1\VERT
			= \VERT z_3^{-1}\circ z_2 \circ z_2^{-1} \circ z_1\VERT
			\\&
			\leq \VERT z_3^{-1} \circ z_2\VERT
				+ \VERT z_2^{-1} \circ z_1\VERT
			= d_\kin(z_2,z_3) + d_\kin(z_1,z_2).
	\end{split}
\ee
One can easily check that $d_\kin$ is symmetric and positive definition, and, hence, it is a metric (see~\cite[Proposition~2.2]{ImbertSilvestre_Schauder}).  Naturally, one defines the kinetic cylinders
\be
	Q_r(z_0) = \{z : t \leq t_0, d_\kin(z,z_0) \leq r\}.
\ee

We will, importantly, consider the distance between a point and a set.  This is defined in the traditional way:
\be
	\dist(S,z) = \inf_{s\in S} d_\kin(s,z).
\ee
It is sometimes useful to use the obvious equality
\be\label{e.dist2}
	\dist(S,z)
		= \inf\left\{ \VERT z'\VERT : z \circ z' \in S\right\}.
\ee
Indeed, for every $s\in S$, we can take $z' = z^{-1}\circ s$, whence $\VERT z'\VERT = d_\kin(s,z)$.

Finally, we define the kinetic convolution:
\be\label{e.convolution}
	(f*g)(z) = \int f(\tilde z) g(\tilde z^{-1} \circ z) \dtz.
\ee
We note that this conflicts with the standard notation for convolution; however, as that does not appear in this article, there is no risk of confusion.  We sometimes convolve $f$ and $g$ where $g$ has no time dependence.  In this case, we abuse notation and denote it the same way.
%\be
%	(f*g)(z) = \int f(\tilde z) g(\tilde z^{-1} \circ z) \dtz
%		= \int f(z\circ \tilde z^{-1}) g(\tilde z) \dtz.
%\ee
We note that, for any $i$,
\be\label{e.c0515601}
	\partial_{v_i} (f*g)
		= f*(\partial_{v_i}g)
	\qquad\text{ and }\qquad
	Y (f*g)= f*(Yg).
\ee

We refer to \cite[Section~3.3]{imbert2020smooth} and~\cite[Section~2]{SilvestreBoundary} for more in-depth discussion.

\section{Statement of the main propositions and proof of the main theorem}\label{s.main}

\subsection{The Poincar\'e, Nash, and outgoing regions}

We first decompose $\H\times \R^d$ into natural subdomains on which different functional inequalities hold. Let
\be\label{e.regions}
	\begin{dcases}
		\cP_R &= \{ (x,v) \in \H\times \R^d: \dist(\R\times \gamma_-, (0,x,v)) \leq \sqrt R\},
		\\
		\cO_R &= \{ (x,v) \in (\H\times \R^d)\setminus \cP_R : \dist(\R\times \gamma_+, (0,x,v)) \leq \sqrt{\sfrac{R}{10}}\}, \quad\tand
		\\
		\cN_R & %= \left(\H\times \R^d\right) \setminus \left( \cP_R \cup \cO_R\right)
			%\\&
			%\subset 
			=\{(x,v) \in \H\times \R^d : \dist(\R\times \partial \H \times \R^d, (0,x,v)) \geq \sqrt{\sfrac{R}{10}}\}.
	\end{dcases}
\ee
The reason for the difference in choice of distance for $\cO_R$ and $\cN_R$ is technical and related to the fact that we want $v_1$ to be bounded away from zero when $(x,v) \in \cO_R$.

Clearly $\cP_R \cup \cO_R\cup \cN_R$ is a decomposition of $\H\times \R^d$.  In the proof we handle the estimates on each set separately.  Along these lines, we require cutoff functions with nice scaling properties for each set.  For this, it is useful to note that
\be\label{e.domain_scaling}
	\cP_R = \delta_{\sfrac{1}{\sqrt R}}P_1,
	\quad
	\cO_R = \delta_{\sfrac{1}{\sqrt R}}O_1,
	\quad\text{ and }\quad
	\cN_R = \delta_{\sfrac{1}{\sqrt R}}N_1.
\ee
It is sometimes helpful to keep in mind that we eventually choose $R = O(t)$.  In this sense, $\cO_R$ and $\cP_R$ are, roughly, the sets in which transport can connect $(x,v)$ to the boundaries $\gamma_-$ and $\gamma_+$, respectively, in time $O(t)$.  The one subtlety is that, for $\cP_R$, we allow ``jumps'' in velocity of size $O(\sqrt R)$, while we use pure transport in $\cO_R$.

%While $\cP_R$ and $\cN_R$ have a nice montonicity property, that is, that
%\be
%	\cP_R \subset P_{R'}
%	\quad \text{ and } \quad
%	N_{R'} \subset \cN_R
%	\qquad\text{ for any } R' > R > 0,
%\ee
%the same is not true for $\cO_R$.  For the convenience of constructing cutoff function, we let
%\be
%	\Theta_R = \{(x,v) \in (\H \times \R^d) \setminus \cP_{\sfrac{R}{4}} : \dist(\R\times \gamma_+) \leq \sqrt R\}.
%\ee
%We note that this satisfies the same scaling properties as in~\eqref{e.domain_scaling}.

\subsection{The main propositions and lemmas}

%The proof of \Cref{t.main} involves combining a general functional inequality, that is, one that holds for any $g\in H^1_\kin$, with a bound on how solutions $f$ to~\eqref{e.kfp} behavior far to the ``bottom right,'' that is, where $x_1 \gg 1$ and $v_1 \ll -1$.  
The proof of \Cref{t.main} involves combining two estimates: the first is a general functional inequality that holds for any $g\in H^1_\kin$, the second is a bound on how solutions $f$ to~\eqref{e.kfp} have far to the ``bottom right,'' that is, where $x_1 \gg 1$ and $v_1 \ll -1$ (the isolated region $\cI_R$ in \Cref{f.regions}).
This latter region is, on the time scale $t\approx R$, isolated from the boundaries, but is not on the infinite time scales on which $\tilde \phi$ and $\phi$ are defined.  We break these into separate estimates at this point because both may be of an independent interest.

Let us state our general functional inequality here.  It arises by, roughly, combining Poincar\'e type inequalities for $x_1 \ll \max\{|v_1|^3,|v_1|\}$ with the Nash inequality (\Cref{p.Nash}) when $x_1 \gg \max\{|v_1|^3, |v_1|\}$.  Its proof is in \Cref{s.functional}.
\begin{Lemma}\label{l.functional}
	Fix $R,\delta >0$.  Suppose that $g\in H^1_{\kin,0}\left( [T_1-2R,T_2]\times \H \times \R^d\right)$ with $T_1 \geq 2R$.  Then
	\be
		\begin{split}
			\int_{T_1}^{T_2} \int &g(z)^2 \dz
				- \delta \int_{T_1 - 2R}^{T_2} \int g(z)^2 \dz
				\\&\lesssim \frac{R}{\delta} \lbr g\rbr_{H^1_\kin([T_1-2R,T_2+R]\times \H \times \R^d)}^2
						+ R \int_{T_1}^{T_2+R} \int_{\R^{d-1}\times \tilde \H} |v_1| g(t,(0, \overline x),v)^2 \, d \overline x \dv dt
						\\&\qquad
						+ \frac{1}{R^{2d+1}} \|g\|_{L^1([T_1-2R,T_2]\times \cN_{\sfrac{R}{2}})}^2.
		\end{split}
	\ee
\end{Lemma}

The first two terms on the right hand side are exactly as we would expect for the energy equality~\eqref{e.energy} associated to solutions of~\eqref{e.kfp}.  The last term, however, is not what we desire because it does not include the steady solution to the adjoint equation $\tilde \phi$.  On a portion of $\cN_R$, we can ``sneak in'' a factor of $\tilde \phi/ R^{\sfrac14}$.  Indeed, using~\eqref{e.phi_asymp}, we can deduce the following lemma, whose proof is in \Cref{s.technical}:
\begin{Lemma}\label{l.cN_R}
	Fix $R>0$.  If $(x,v) \in \cN_R$ and $x_1 \geq - v_1^3$, then
\be
	\tilde \phi(x,v) \gtrsim R^{\sfrac14}.
\ee
\end{Lemma}
Thus, on this subdomain, we can always replace the $L^1$-norm of $g$ with $R^{-\sfrac14}\|g \phi^*\|_{L^1}$, which is a quantity conserved by the equation~\eqref{e.kfp}.

On the other hand, when $x \leq - v^3$, we have no lower bound on $\phi^*$ and can not appeal to $\|f\phi^*\|_{L^1}$.  This region is ``too far'' from the boundary to be influenced by it on a time-scale $t= O(R)$.  Thus, we have the following estimate that quantifies how isolated it is.

\begin{Lemma}\label{l.isolated_estimate}
	Fix $R>0$.  Let $f$ be a solution to~\eqref{e.kfp}.  There exists a nonnegative function $\tilde \mu_R \lesssim 1$ such that
	\be\label{e.c061804}
	\begin{split}
		&\tilde \mu_R(x,v)
			= 1 \qquad \text{ if } (x,v) \in \cN_R \cap \{x_1 \leq - v_1^3\},
		\\
		&\tilde \mu_R(x,v)
			\lesssim \begin{dcases}
					0 \qquad
						&\text{ if } v_1> - \frac{1}{2}\sqrt{\tfrac{R}{10}} \text{ or }  x_1 \geq 2 |v_1|^3,\\
					e^{\frac{Rv_1}{10x_1}} \qquad
						&\text{ if }  v_1 \leq -\sqrt{\tfrac{R}{10}} \text{ and } x_1 \leq -\frac{R v_1}{10},
				\end{dcases}
			\end{split}
	\ee
	and
	\be\label{e.c061805}
			(\Delta_v + v\cdot\nabla_x) \tilde \mu_R \lesssim \frac{1}{R^{\sfrac54}} \tilde \phi.
	\ee
%	\be\label{e.isolated_region}
%		\int_{\cN_R \cap\{x \leq -v_1^3\}} f(t,x,v) \dv dx
%			-  \int \chi_R(x,v) f_0(x,v) \dv dx
%			\lesssim \frac{t}{R^{\sfrac54}} \int \tilde \phi(x,v) f_0(x,v) \dv dx.		
%	\ee
\end{Lemma}

%Let us note that the main step of \Cref{l.isolated_estimate} is the construction of the cutoff function.  
As discussed in the introduction, the construction of this cutoff-type function requires some care as it has to encode the physics of the situation -- particles in the $x \approx - R v_1$ region will exit the region $\cN_R \cap \{x_1 \leq -v_1^3\}$ in $R$ units of time.  That said, the proof of \Cref{l.isolated_estimate} is rather tedious, so we relegate it to \Cref{s.zeta}.

Before continuing on, let us note that the $\sfrac{R}{10}$ is somewhat arbitrary.  It comes from the $\sfrac{R}{10}$ taken in the definition of $\cN_R$, which is mainly taken for convenience.  This can certainly be improved, although it is not clear exactly what the optimal exponential decay rate is.

We now combine all estimate into one that will be the main functional inequality in the proof of \Cref{t.main}.

\begin{Proposition}\label{p.combined_estimate}
	Fix $T_1,T_2, R, \delta >0$ with $T_2 > T_1 > 2R$.  Suppose that $f\in H^1_{\kin,0}\left( [0,T_2]\times \H \times \R^d\right)$ solves~\eqref{e.kfp}.  Then
	\be
		\begin{split}
			\int_{T_1}^{T_2} \int &f(z)^2 \dz
				- \delta \int_{T_1 - 2R}^{T_2} \int f(z)^2 \dz
				\\&\lesssim \frac{R}{\delta} \lbr f\rbr_{H^1_\kin([T_1-2R,T_2+R]\times \H \times \R^d)}^2
						+ R \int_{T_1}^{T_2+R} \int_{\R^{d-1}\times \tilde \H} |v_1| f(t,(0, \overline x),v)^2 \, d \overline x \dv dt
						\\&\qquad
						+ \Big(\frac{(T_2 - T_1)^4}{R^{2d + \sfrac72}}+\frac{1}{R^{2d - \sfrac12}}\Big) \Big( \int f_\init \tilde \phi \dx\dv\Big)^2
%						\\&\qquad
						+ \frac{(T_2-T_1)^2}{R^{2d+1}} \Big(\int f_\init  \tilde \mu_R \dx\dv\Big)^2.
		\end{split}
	\ee
\end{Proposition}
\begin{proof}
For convenience, let us write
\be
	\cN_{\sfrac{R}{2}}
		= \cW_{\sfrac{R}{2}} \cup \cI_{\sfrac{R}{2}},
\ee
where
\be
	\cI_{\sfrac{R}{2}} = \{(x,v) \in \cN_{\sfrac{R}{2}} : \ v_1 \leq 0, \ x_1 \leq - v_1^3\}
	\quad\text{ and }\quad
	\cW_{\sfrac{R}{2}} = \cN_{\sfrac{R}{2}} \setminus \cI_{\sfrac{R}{2}}.
\ee
Here $\cI_{\sfrac{R}{2}}$ is the ``isolated region,'' where the effects of the boundary have not ``yet'' been felt.  In this region, we use \Cref{l.isolated_estimate}.  Its complement, $\cW_{\sfrac{R}{2}}$, is the ``weighted region'', where $\tilde \phi$ can be included directly in the integral via \Cref{l.cN_R}.

First, we note that, by \Cref{l.isolated_estimate},
\be
	\begin{split}
		\frac{d}{dt} \int f(t,x,v) \tilde\mu_R(x,v) \dx\dv
			&= \int \left[\left( \Delta_v - v\cdot\nabla_x\right) f(t,x,v)\right] \tilde \mu_R \dx\dv
			\\&
			= \int  f(t,x,v) \left( \Delta_v + v\cdot\nabla_x\right) \tilde \mu_R \dx\dv
			\\&
			\lesssim \frac{1}{R^{\sfrac54}}\int  f(t,x,v) \tilde \phi(x,v) \dx\dv
			= \frac{1}{R^{\sfrac54}}\int  f_\init (x,v) \tilde \phi(x,v) \dx\dv.
	\end{split}
\ee
The last equality holds by~\eqref{e.tilde_phi_conserved}.  We deduce that
\be
	\begin{split}
		\int_{\cI_{\sfrac{R}{2}}} f(t,x,v) \dx\dv
			&\leq \int f(t,x,v) \tilde \mu_R(x,v)  \dx\dv
			\\&
			\lesssim \frac{t}{R^{\sfrac54}} \int f_\init \tilde \phi \dx\dv
				+ \int f_\init (x,v) \tilde \mu_R(x,v)  \dx\dv.
	\end{split}
\ee
Next, we use \Cref{l.cN_R} to find
\be
	\int_{\cW_{\sfrac{R}{2}}} f(t,x,v) \dx\dv
		\lesssim \frac{1}{R^{\sfrac14}}\int_{\cW_{\sfrac{R}{2}}} f(t,x,v) \tilde \phi \dx\dv
		= \frac{1}{R^{\sfrac14}}\int_{\cW_{\sfrac{R}{2}}} f_\init \tilde \phi \dx\dv.
\ee
In total, we deduce that
\be\label{e.c062009}
	\begin{split}
		\|f\|_{L^1([T_1-2R,T_2]\times \cN_{\sfrac{R}{2}})}
			\lesssim &\frac{(T_2 - T_1)^2}{R^{\sfrac54}} \int f_\init \tilde \phi \dx\dv
					+ (T_2-T_1) \int f_\init (x,v) \tilde \mu_R(x,v)  \dx\dv
					\\&
					+ R^{\sfrac34} \int f_\init \tilde \phi \dx\dv.
	\end{split}
\ee
The combination of this inequality with \Cref{l.functional} completes the proof.
\end{proof}

\subsection{Proof of the main result: \texorpdfstring{\Cref{t.main}}{Theorem \ref{t.main}}}

\begin{proof}
The proof takes several steps.  All but the last aim for a weighted $L^1\to L^2$ type estimate.  The last step bootstraps that to a weighted $L^2 \to L^\infty$ type estimate.

\medskip
\noindent{\bf \# Step one: setting notation.} 
For ease, let us denote the ``energy'' and ``dissipation'' as
\be\label{e.definitionED}
	E(t) = \int f(t,x,v)^2 \dx\dv
	\quad\tand\quad
	D(t) = \int |\nabla_v f(t,x,v)|^2 \dx\dv.
\ee
Although physically it is not correct to call $E$ the energy, we abuse terminology and do so in analogy with work for parabolic equations.  It is also useful to set notation for the boundary term
\be
	B(t) = \int_{\gamma_+} |v_1| f(t,(0,\overline x), v) \, d\overline x \dv.
\ee

Applying~\eqref{e.energy} yields, for any nonnegative $t_1 < t_2$,
\be\label{e.c061902}
	E(t_2) + \int_{t_1}^{t_2} \left( D(s) + B(s)\right) \ds
		\leq E(t_1).
\ee
We see that $E$ is decreasing.

\medskip
\noindent{\bf \# Step two: applying the weighted Nash inequality \Cref{p.combined_estimate}.} 
With $\delta, \epsilon \in (0,\sfrac1{100})$ be to chosen, we let
\be
	R = \eps t,
	\quad
	T_1 -2R = \sfrac{t}{2},
	\quad\tand\quad
	T_2 + R = t.
\ee
Then, in the notation above and in view of the correspondence~\eqref{e.H1_kin_f} between $D$ and the $H^1_\kin$-norm,  \Cref{p.combined_estimate} yields
\be
	\begin{split}
		\int_{\frac{1+4\eps}{2} t}^{(1-\eps) t} &E(s) \ds
			- \delta \int_{\sfrac{t}{2}}^{(1-\eps) t} E(s) \ds
			\\&
			\lesssim \frac{\eps t}{\delta} \int_{\sfrac{t}{2}}^t \left(D(s) + B(s)\right) \ds
				+ \frac{1}{\eps^{2d+\sfrac72} t^{2d-\sfrac12}} \Big( \int f_\init \tilde \phi \dx\dv\Big)^2
			\\&\qquad
				+ \frac{1}{\eps^{2d+1} t^{2d-1}} \Big(\int f_\init \tilde \mu_{\eps t} \dx\dv\Big)^2.
	\end{split}	
\ee
Combining this with~\eqref{e.c061902}, we see that
\be\label{e.c061903}
	\begin{split}
		\int_{\frac{1+4\eps}{2} t}^{(1-\eps) t} &E(s) \ds
			- \delta \int_{\sfrac{t}{2}}^{(1-\eps) t} E(s) \ds
			\\&
			\lesssim \frac{\eps t}{\delta} E(\sfrac{t}{2})
				+ \frac{1}{\eps^{2d+\sfrac72} t^{2d-\sfrac12}} \Big( \int f_\init \tilde \phi \dx\dv\Big)^2
%			\\&\qquad
				+ \frac{1}{\eps^{2d+1} t^{2d-1}} \Big(\int f_\init \tilde \mu_{\eps t} \dx\dv\Big)^2.
	\end{split}	
\ee

\medskip
\noindent{\bf \# Step three: setting up a ``first touching'' argument.}
Fix $\bar\alpha, \overline \beta \gg 1$ be constants to be chosen, and let
\be\label{e.c070702}
	\alpha = \overline \alpha \Big( \int f_\init \tilde \phi \dx\dv\Big)^2
	\quad\tand\quad
	\beta = \overline \beta \Big(\int f_\init \tilde \mu_{\eps t_0} \dx\dv\Big)^2.
\ee
Define
\be
	t_0 = \sup \Big\{ t : E(s) \leq \frac{\alpha}{t^{2d+\sfrac12}} + \frac{\beta}{t^{2d}}\Big\}.
\ee
Up to approximation, we may assume that $f_\init$ is smooth and compactly supported, whence $t_0 > 0$.  Our goal is to show that $t_0 = \infty$.  Hence, we argue by contradiction assuming that $t_0$ is finite.

Let us note that, we immediately have, from the definition of $t_0$,
\be\label{e.c070701}
	E(s) \lesssim \frac{\alpha}{t^{2d+\sfrac12}} + \frac{\beta}{t^{2d}}
		\qquad\text{ for all } s\in [\sfrac{t_0}{4},t_0].
\ee
We use this frequently in the sequel.

\medskip
\noindent{\bf\# Step four: obtaining a contradiction to the definition of $t_0$.} 
Moving the negative integral term on the left hand side of~\eqref{e.c061903} to the right hand side and applying the definition~\eqref{e.c070701} of $t_0$, we deduce that
\be\label{e.c061904}
	\begin{split}
		t_0 E\left((1-\eps) t_0\right)
			&
			\lesssim \frac{\eps t}{\delta} E(\sfrac{t}{2})
				+ \delta \int_{\sfrac{t}{2}}^{(1-\eps)t} E(s) \ds
				+ \frac{1}{\eps^{2d+\sfrac72} t^{2d-\sfrac12}} \Big( \int f_\init \tilde \phi \dx\dv\Big)^2
%			\\&\qquad
				\\&\qquad
				+ \frac{1}{\eps^{2d+1} t^{2d-1}} \Big(\int f_\init \tilde \mu_{\eps t} \dx\dv\Big)^2
			\\&\lesssim
				\left(\frac{\eps}{\delta} + \delta\right)\frac{\alpha}{t_0^{2d-\sfrac12}}
				+ \left(\frac{\eps}{\delta} + \delta \right)\frac{\beta}{t_0^{2d-1}}
			\\&\qquad
				+ \frac{1}{\eps^{2d+\sfrac72} t_0^{2d-\sfrac12}} \Big( \int f_\init \tilde \phi \dx\dv\Big)^2
				+ \frac{1}{\eps^{2d+1} t_0^{2d-1}} \Big(\int f_\init \tilde \mu_{\eps t_0} \dx\dv\Big)^2.
	\end{split}
\ee
Recalling the definition~\eqref{e.c070702} of $\alpha$ and $\beta$, we find
\be\label{e.c061904a}
	\begin{split}
		E\left((1-\eps) t_0\right)
			&\leq  \frac{\alpha}{t_0^{2d+\sfrac12}}
				\left(\frac{\eps}{\delta} +\delta
					+ \frac{1}{\overline \alpha \eps^{2d+\sfrac72}}
				\right)
			+ \frac{\beta}{t_0^{2d}}
				\left(\frac{\eps}{\delta} + \delta
					+ \frac{1}{\overline \beta \eps^{2d+1}}
				\right).
	\end{split}
\ee
Recalling again that $E$ is decreasing, we have
\be
	E((1-\eps) t_0)
		\geq E(t_0)
		= \frac{\alpha}{t_0^{2d+\sfrac12}}
			+ \frac{\beta}{t_0^{2d}}.
\ee
In summary,
\be
	\frac{\alpha}{t_0^{2d+\sfrac12}}
			+ \frac{\beta}{t_0^{2d}}
			\leq \frac{\alpha}{t_0^{2d+\sfrac12}}
				\left(\frac\eps\delta +\delta
					+ \frac{1}{\overline \alpha \eps^{2d+\sfrac72}}
				\right)
			+ \frac{\beta}{t_0^{2d}}
				\left(\frac\eps\delta +\delta)
					+ \frac{1}{\overline \beta \eps^{2d+1}}
				\right).
\ee
Choosing $\eps$ small, then $\delta$ small (depending on $\eps$), and then choosing $\bar\alpha$ and $\overline \beta$ large (depending on both $\eps$ and $\delta$, we obtain a contradiction.

It follows that $t_0 = \infty$ and, hence, for all $t>0$,
\be\label{e.c061905}
	E(t)
		\leq \frac{\bar\alpha}{t^{2d+\sfrac12}} \Big( \int f_\init \tilde \phi \dx\dv\Big)^2
			+ \frac{\bar\beta}{t^{2d}} \Big( \int f_\init \tilde \mu_{\eps t} \dx\dv\Big)^2.
\ee

\medskip
\noindent
{\bf \# Step five: some functional analysis and the conclusion.}
The inequality~\eqref{e.c061905} implies that the solution operator of~\eqref{e.kfp}
\be
	S_t: X_t \to L^2(\H \times \R^d)
\ee
is well-defined and bounded.  In other words, $S_t f_\init = f(t)$ if $f_\init \in X_t$.  Here, we define the Banach space
\be
	X_t
		= L^1_{\tilde \phi} \cap \left(t^{-\sfrac14}L^1_{\tilde \mu_{\eps t}}\right)
		= \Big\{ h \in L^1_{\rm loc}(\H \times \R^d)
			: \int |h| (\tilde \phi + \tilde \mu_{\eps t}) \dx \dv < \infty\Big\}
\ee
with the norm
\be
		\|h\|_{X_t}
			= \int  |h| \tilde \phi \dx \dv
				+ t^{\sfrac14} \int |h| \tilde \mu_{\eps t} \dx \dv.
\ee
Hence,~\eqref{e.c061905} translates to the bound
\be\label{e.c061906}
	\|S_t\|_{X_t \to L^2} \lesssim \frac{1}{t^{d+\sfrac14}}.
\ee

By using the fact that $\tilde f$ solves~\eqref{e.akfp} if $f$ solves~\eqref{e.kfp}, we also obtain the bound
\be
	\|\tilde S_t\|_{\tilde X_t \to L^2} \lesssim \frac{1}{t^{d+\sfrac14}}.
\ee
where 
\be
	\tilde S_t: \tilde X_t \to L^2(\H \times \R^d)
\ee
is the solution operator of~\eqref{e.akfp} and
\be
	\tilde X_t
		= L^1_{\phi} \cap \left(t^{-\sfrac14}L^1_{\mu_{\eps t}}\right)
		= \Big\{ h \in L^1_{\rm loc}(\H \times \R^d)
			: \int |h| (\phi + \mu_{\eps t}) \dx \dv < \infty\Big\}
\ee
with the norm
\be
		\|h\|_{\tilde X_t}
			= \int  |h| \phi \dx \dv
				+ t^{\sfrac14} \int |h| \mu_{\eps t} \dx \dv.
\ee
Let us note that since~\eqref{e.kfp} and~\eqref{e.akfp} are adjoint to one another,
\be
	\tilde S_t^* : L^2 \to X_t^*
\ee
is also a solution operator to~\eqref{e.kfp}.  By standard results on adjoint operators, we deduce that
\be\label{e.c061907}
	\|\tilde S_t^*\|_{L^2 \to \tilde X_t^*}
		= \|\tilde S_t\|_{\tilde X_t \to L^2}
		\lesssim \frac{1}{t^{d+\sfrac14}}.
\ee
It is easy to identify $\tilde X_t^*$ as
\be
	\tilde X_t^*
		= L^\infty_{\sfrac1\phi}
			+ t^{\sfrac14} L^\infty_{\sfrac{1}{\mu_{\eps t}}}
		= \{h : h = \phi h_1 + t^{\sfrac14} \mu_{\eps t} h_2,\ h_i \in L^\infty\},
\ee
with the norm
\be
	\|h\|_{\tilde X_t^*}
		= \inf_{h = \phi h_1 + t^{\sfrac14} \mu_{\eps t} h_2} \left(\|h_1\|_{L^\infty} + \|h_2\|_{L^\infty}\right).
\ee

We now conclude using the semigroup property
\be
	f(t) = \tilde S_{\sfrac{t}{2}}^* S_{\sfrac{t}{2}} f_\init.
\ee
Indeed, recalling~\eqref{e.c061906} and~\eqref{e.c061907},
\be
	\|f(t)\|_{\tilde X_{\sfrac{t}{2}}^*}
		= \|\tilde S_{\sfrac{t}{2}}^* S_{\sfrac{t}{2}} f_\init\|_{\tilde X_t^*}
		\lesssim \frac{1}{t^{d + \sfrac14}} \|S_{\sfrac{t}{2}} f_\init\|_{L^2}
		\lesssim \frac{1}{t^{2d + \sfrac12}} \|f_\init\|_{X_{\sfrac{t}{2}}}.
\ee
The proof is finished after unpacking the definitions of the norms.
\end{proof}

\section{The main functional inequality: \texorpdfstring{\Cref{l.functional}}{Lemma \ref{l.functional}}}
\label{s.functional}

We decompose $\H\times \R^d$ into three regions $\cP_R$, $\cN_R$, and $\cO_R$, depending on the influence of the boundary.  Recall that these are defined in~\eqref{e.regions}.  We state $L^2$-estimates on each regime, postponing their proofs until \Cref{s.functional_proofs}.  In \Cref{s.functional_proof}, we combine these estimates to prove \Cref{l.functional}.

\subsection{The functional inequalities on each region}

We begin by stating the Poincar\'e-type inequality.  This inequality is on the portion of the domain that is ``close'' to the boundary $\gamma_-$, where particles are absorbed.  It quantifies this effect.
\begin{Proposition}
	[Poincar\'e inequality on the incoming region $\cP_R$]
	\label{p.poincare}
Fix any positive numbers $R$, $T_1$, and $T_2$ such that $2R < T_1 < T_2$.  Suppose that $g\in H^1_{\kin,0}([T_1- 2R,T_2]\times\H\times \R^d)$. Then
\be\label{e.c030109}
	\begin{split}
		\|g\|_{L^2([T_1,T_2]\times \cP_R)}^2
			\lesssim
				&R\lbr g\rbr_{H^1_\kin([T_1 - 2 R,T_2]\times \H \times \R^d)}^2
				\\&
				+ \sqrt R \lbr g\rbr_{H^1_\kin([T_1 - 2 R,T_2]\times \H \times \R^d)}
				\|g\|_{L^2([T_1 - 2 R,T_2]\times \H \times \R^d)}.
	\end{split}
\ee
\end{Proposition}
Let us note that the norms on the right hand side can be localized to $P_{cR}$, for an appropriate $c>1$, with some extra care in the proof.  We opt for simplicity here.

Next, we state the Nash-type inequality.  This inequality is on the portion of the domain that is ``far'' from all boundaries.  It quantifies the fact that the evolution of~\eqref{e.kfp} is on $\R \times \H\times \R^d$ is essentially the same it would be on $\R \times \R^d \times \R^d$.  Let us make note that $\cN_{\sfrac{R}2}$ is {\em larger} than $\cN_R$.
\begin{Proposition}
	[Nash inequality on $\cN_R$]
	\label{p.Nash}
Suppose that $g\in H^1_\kin\left( [T_1,T_2]\times \H \times \R^d\right)$, and fix $R,\eps>0$.  If $\eps$ is sufficiently small,
\be
	\begin{split}
		\int_{T_1}^{T_2} \|g\|_{L^2(\cN_R)}^2\,dt
			\lesssim & \sqrt{\eps R} \lbr g\rbr_{H^1_\kin([T_1-\eps R,T_2]\times \cN_{\sfrac{R}{2}})} \| g\|_{L^2([T_1-\eps R,T_2]\times \cN_{\sfrac{R}{2}})}
				\\&\qquad
				+ \frac{1}{(\eps R)^{2d+1}} \|g\|_{L^1([T_1-\eps R,T_2]\times \cN_{\sfrac{R}{2}})}^2.
	\end{split}
	\ee
\end{Proposition}

Finally, we state Poincar\'e-type inequality of a different flavor.  Particles near the outgoing part of the boundary $\gamma_+$, will likely leave the domain.  In the context of~\eqref{e.kfp}, this quantifies the effect of the boundary in~\eqref{e.energy}.  The following estimate quantifies that.
\begin{Proposition}
	[Poincar\'e inequality on the outgoing region $\cO_R$]
	\label{p.outgoing}
Fix any positive numbers $R$, $T_1$, and $T_2$ such that $2R < T_1 < T_2$.  Suppose that $g\in H^1_\kin([T_1- 2R,T_2]\times\H\times \R^d)$. Then
\be
	\begin{split}
		\|g\|_{L^2([T_1,T_2]\times \cO_R)}^2
			\lesssim
				&R\lbr g\rbr_{H^1_\kin([T_1,T_2+R]\times \H\times \R^d)}^2
				\\&
				+ \sqrt R \lbr g\rbr_{H^1_\kin([T_1,T_2 + R]\times \H\times \R^d)}
			\|g\|_{L^2([T_1,T_2 + 2R]\times\H\times \R^d))}
			\\&
				+ R\int_{T_1}^{T_2 + R} \int_{\R^d} \int_{\R^{d-1}}
					|v_1| f(t,(0,\overline x), v)^2 \, d\overline x \dv \dt.
	\end{split}
\ee
\end{Proposition}

As we noted after \Cref{p.poincare}, the norms on the right hand side of the inequality in \Cref{p.outgoing} can be localized with some extra care.

We remind the reader that the proofs of these lemmas can be found in \Cref{s.functional_proofs}.

\subsection{The proof of \texorpdfstring{\Cref{l.functional}}{Lemma \ref{l.functional}}}\label{s.functional_proof}

\begin{proof}
	It is clear that this is a simple consequence of \Cref{p.poincare,p.Nash,p.outgoing} combined with Young's inequality.  We omit the details.
\end{proof}

\subsection{Establishing the Nash and Poincar\'e-type inequalities}\label{s.functional_proofs}

\subsubsection{The proof of the Poincar\'e-type inequality on $\cP_R$}

To begin, we first show that the Poincar\'e regime can be characterized in a simple way, depending on $(x,v)$.  This is useful in understanding ``paths'' from any point in $\cP_R$ to the boundary $\gamma_+$ that are at the heart of the proof of \Cref{p.poincare}.  The proof is postpone to \Cref{s.technical}.

\begin{Lemma}\label{l.cP_R_bounds}
	Suppose that $(x,v) \in \cP_R$.  Then
	\be
		x_1 \leq R \max\{v_1,3\sqrt R\}
		\quad\tand\quad
		-2 \sqrt R \leq v_1.
	\ee
	Additionally, if $(x,v) \in \cO_R$, then
	\be
		\frac{x_1}{|v_1|} \leq R.
	\ee
\end{Lemma}

We now proceed with our Poincar\'e-type estimate.  For the proof of \Cref{p.poincare}, let us make the convention that every norm is taken on $[T_1-2R,T_2]\times \H\times \R^d$ unless otherwise specified.  For example, by writing
\be
	\|g\|_{L^2}
		\quad\text{ we mean }\quad
	\|g\|_{L^2([T_1-2R,T_2]\times \H\times \R^d)}.
\ee
This saves significant space and does not cost clarity.

\begin{proof}[Proof of \Cref{p.poincare}]
Let us extend $g$ to $\overline g$ by
\be
	\overline g(t,x,v)
		= \begin{dcases}
			g(t,x,v) \qquad &\text{ if } x_1>0,\\
			0 \qquad &\text{ if } x_1 \leq 0 < v_1.
		\end{dcases}
\ee
Take any mollifier: a nonnegative, smooth function $\psi$ such that
\be
	\int \psi(z) \dz = 1.
\ee
Up to translation and dilation, we may assume that
\be
	\supp \psi \subset (0,1) \times \H \times \tilde \H
		= \{t \in (0,1), x_1 >0, v_1 < 0\}.
\ee
Define, for any $s\in (0,\sqrt{\eps R}]$, 
	\be
	\begin{split}
		g_s(z)
		& = (\psi_s*\overline g)(z)
		= \int \psi(\delta_{\sfrac1s}(\tilde z)) \overline g(\tilde z^{-1} \circ z)  \frac{\dtz}{s^{4d+2}}
		= \int \psi(\tilde z) \overline g(\delta_s(\tilde z)^{-1}\circ z) \dtz,
	\end{split}	
	\ee
	where, for all $z$,
	\be
	\psi_s(z) := \frac{1}{s^{4d+2}}\psi\left(\delta_{\sfrac1s}(z)\right)
	\qquad\text{ satisfies }
	\quad
	\int \psi_s \dw\dy\ds = 1
	\ee
	We observe a few simple facts about $g_s$.  First, after changing variables, we see that $g_s$ is smooth. Second, from~\eqref{e.c0515601}, it is clear that
\be
	\lim_{s\to0}\|g_s - g\|_{H^1_\kin} = 0.
\ee
Hence, we need only prove~\eqref{e.c030109} for $g_s$.
%	by~\eqref{e.c050302} and a standard change of variables. Next, by fundamental theorem of calculus, we see
%	\be\label{e.w05111}
%	\begin{split}
%		 g_{\sqrt{\eps R}}(z)^2
%		= &
%		- 2\int_0^{\eps R}  g_{\sqrt{\eps R},b,0}(z) \partial_b g_{\sqrt{\eps R},b,0}( z)  \,db
%		\\&
%		+ 2\int_{\sqrt{\eps R}}^0 g_{\sqrt{\eps R},\eps R,c}(z) \partial_c g_{\sqrt{\eps R};\eps R,c}(z)   \,dc
%	\end{split}
%	\ee
Finally, due to the choice of support of $\psi$, we see that $g_s$ is well-defined on
\be
	[T_1-\eps R, T_2]\times \Gamma
		:= [T_1-\eps R, T_2]\times
		%\underbrace{\left(\left(\H \times \R^d\right) \cup \left( \tilde \H \times \H\right)\right)}_{=
		\{\text{$x_1 \geq 0$ or both $x_1\leq 0$ and $v_1\geq0$}\},%}%\left( \left( \H \times \R^d \right) \cup \{x_1 < 0, v_1 > 0\} \right)
\ee
and that
%This is smooth and well-defined on $(\R_+ \times \R) \cup (\R \times \R_+)$ by the choice of support of $\psi$. 
% Moreover,
\be\label{e.c030108}
	g_s(t,x,v) = 0 \qquad \text{ for any } x_1 \leq 0, v_1 \geq 0.
\ee
We use the notation $\Gamma$ here because this region, when $d=1$, looks approximately like a ``backwards'' $\Gamma$.  See \Cref{f.Gamma}.

\begin{figure}
\centering
\begin{tikzpicture}[scale=.5]
% Draw axes

% Shade quadrants
\fill[gray!50] (0,0) -- (3,0) -- (3,3) -- (0,3) -- cycle;  % First quadrant
\fill[gray!50] (0,0) -- (-3,0) -- (-3,3) -- (0,3) -- cycle;  % Second quadrant
\fill[gray!50] (0,0) -- (3,0) -- (3,-3) -- (0,-3) -- cycle;  % Fourth quadrant

\draw[<->] (-3,0) -- (3,0) node[right] {\tiny $x_1$};
\draw[<->] (0,-3) -- (0,3) node[above] {\tiny $v_1$};
\end{tikzpicture}
\caption{The region $\Gamma$.  Notice that it looks like a backwards $\Gamma$.}
\label{f.Gamma}
\end{figure}
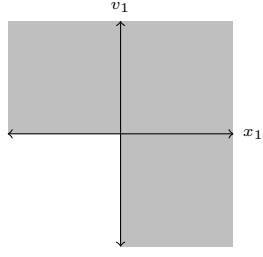

%\begin{comment}
%Finally, using the fundamental theorem of calculus and the definition of $\Omega_{2R}$, if $(x,v) \in \Omega_{2R}$,
%\be\label{e.c030107}
%	\begin{split}
%		g_\eps(t,x,v)^2
%			= &- 2 \int_0^{\sqrt{4R}} g_\eps(t,x,v+w) \partial_v g_\eps(t,x,v+w) \dw
%				\\&
%				- 2 \int_0^{4R} g_\eps(t-s,x-sv_R,v_R) (Y g_\eps)(t-s, x-sv_R, v_R) \ds,
%	\end{split}
%\ee
%where 
%\end{comment}

Let $\chi_{\cP_R}$ be a cutoff function for $\cP_R$ such that
\be\label{e.chi_cP_R1}
	\chi_{\cP_R} \equiv 1 \qquad \text{ in } \cP_R
	\quad\text{ and }\quad
	\chi_{\cP_R} \equiv 0 \qquad \text{ in } \cP_{\sfrac{3R}{2}}^c
\ee
while
\be\label{e.chi_cP_R2}
	\|\nabla_v \chi_{\cP_R}\|_{L^\infty} \lesssim \frac{1}{\sqrt R}
	\quad\text{ and }\quad
	\|\nabla_x \chi_{\cP_R}\|_{L^\infty} \lesssim \frac{1}{R^{\sfrac32}}.
\ee
This can easily be constructed when $R=1$ and the general case follows by letting
\be\label{e.chi_cP_R3}
	\chi_{\cP_R}(z) = \chi_{P_1}\left(\delta_{1/\sqrt R} (z)\right).
\ee
Notice that we use the scaling~\eqref{e.domain_scaling}.

Fix any $(x,v) \in\supp \chi_{\cP_R}$, and let
\be
	v_R = v+ 10 e_1 \sqrt{R}
\ee
for succinctness.  Here $e_1 = (1,0,\dots, 0)$ is the first canonical basis vector.  Let us note that, if $s$ is sufficiently small in a way depending only on $\eps$ and $R$, 
\be\label{e.w05271}
	g_s(x - 2R v_R, v_R) = 0
\ee
because, recalling that $(x,v) \in \cP_{\sfrac{3R}{2}}$ and using \Cref{l.cP_R_bounds},
\be
	x_1 - 2R v_1 - 20 R^{\sfrac32}
		< 0
		< v_1+ 10\sqrt{R}.
\ee
Hence, we may write
\be
	 g_s(t,x,v)^2
	 	%\\&\quad
		= -2  \int_0^{10\sqrt{R}} (g_s\, \partial_{v_1} g_s)(t,x,v+re_1) \dr 
			-2 \int_0^{2R} (g_s \, Yg_s)(t-s, x-rv_R, v_R) \dr.
\ee
We deduce that
\be\label{eq:I1}
\begin{split}
	&\int_{T_1}^{T_2}\int_{\cP_R} g_s(t,x,v)^2 \dz
		\leq \int_{T_1}^{T_2}\int_{\H\times \R^d} g_s(t,x,v)^2 \chi_{\cP_R}(x,v)^2 \dz
		\\&\quad
		= -2 \int_{T_1}^{T_2}\int_{\H\times \R^d}
			\Big( \int_0^{10\sqrt{R}} (g_s\, \partial_{v_1} g_s)(t,x,v+re_1) \dr \Big) \chi_{\cP_R}(x,v)^2 \, dt \, dx \, dv
		\\&\qquad
			-2 \int_{T_1}^{T_2}\int_{\H\times \R^d}  \Big(\int_0^{2R} (g_s \, Yg_s)(t-s, x-rv_R, v_R) \dr \Big) \chi_{\cP_R}(x,v)^2 \, dt \, dx \, dv
		\\&\quad
		=: I_1 + I_2.
\end{split}
\ee
We estimate each term in turn.

Let us handle $I_1$ first as it is simpler.  Then
\be\label{e.c030502}
	\begin{split}
		\left|I_1\right|
			& \approx \Big|\int_0^{10\sqrt{R}} \int_{T_1}^{T_2}\int_{\H\times \R^d}  (g_s\, \partial_{v_1} g_s)(t,x,v)  \chi_{\cP_R}(x,v-re_1)^2 \dt\dx \dv\dr\Big|
			\\&
			\leq \int_0^{10\sqrt{R}} \|\partial_v g_s\chi_{\cP_R}(\cdot,\cdot-re_1)\|_{L^2([T_1,T_2]\times \H\times \R^d)} \|g_s \chi_{\cP_R}(\cdot,\cdot-re_1)\|_{L^2([T_1,T_2]\times \H\times \R^d)}\dr
			\\&
%			\lesssim \int_0^{10\sqrt{R}} \|\partial_v g_s\|_{L^2([T_1,T_2]\times P_{2R})}
%				\|g_s\|_{L^2([T_1,T_2]\times P_{2R})}\dr
			\lesssim \sqrt{R} \lbr g_s\rbr_{H^1_\kin}
				\|g_s\|_{L^2}.
	\end{split}
\ee
%In the second inequality, we used $\chi_{\cP_R}(\cdot,\cdot-re_1) \lesssim \chi_{\cP_R}$, which follows from the monotonicity of $\chi_{\cP_R}$ discussed below~\eqref{e.chi_cP_R2}. 

We now consider $I_2$.  We first change the order of integration and change variables:
\be
	\begin{split}
		-\frac{1}{2}I_2
			&= \int_0^{2R} \int_{T_1}^{T_2}\int_{\H\times \R^d}(g_s\, Yg_s)(t-r, x-rv_R, v_R) \chi_{\cP_R}(x,v)^2 \dt \dx \dv \dr
			\\
			&= \int_0^{2R} \int_{T_1-r}^{T_2-r}\int_{\H\times \R^d} (g_s\, Yg_s)(t, x, v) \chi_{\cP_R}(x+rv_R, v - 10 \sqrt R e_1)^2 \dt \dx \dv \dr.
	\end{split}
\ee
In the second equality, we used that $g_s(t,x,v) \equiv 0$ for $x_1 < 0 < v_1$ and $\chi_{\cP_R}(x,v- 10\sqrt R e_1) \equiv 0$ if $v_1 \leq 0$ (see \Cref{l.cP_R_bounds} and~\eqref{e.chi_cP_R1}).  
We now use that $H^{-1}_v$-$H^1_v$ pairing of $Yg_s$ with $g_s \chi_{\cP_R}$.  We find
\be\label{e.c030501}
	\begin{split}
		|I_2|
			&\lesssim \int_0^{2R} \lbr g_s\rbr_{H^1_\kin}  \| \nabla_v\left( g_s \chi_{\cP_R}(\cdot + r v_R, \cdot - 10\sqrt{R}e_1)\right)\|_{L^2([T_1,T_2]\times \H\times \R^d)} \dr
			\\&
			\lesssim \lbr g_s\rbr_{H^1_\kin}
%				\\&\qquad	\cdot 
				\int_0^{2R} \left(
					\| \nabla_v g_s\|_{L^2}
					+ \|g_s\|_{L^2}
					\left( \frac{r}{R^{\sfrac32}} + \frac{1}{\sqrt R}\right)\right)
					 \dr
			\\&
			\lesssim R \lbr g_s\rbr_{H^1_\kin}^2
				+ \sqrt R \lbr g_s\rbr_{H^1_\kin} \|g_s\|_{L^2}.
	\end{split}
\ee
where the second inequality follows from a simple computation of $\|\nabla_v\chi_{\cP_R}(\cdot + s v_R,v_R)\|_{L^\infty}$  using \eqref{e.chi_cP_R2}.
% and the fact that
%\be
%	\chi_{\cP_R}(\cdot + s v_R,\cdot - 10\sqrt{R}e_1) \lesssim \chi_{\cP_R}.
%\ee
%This follows from the monotonicity of $\chi_{\cP_R}$ discussed below~\eqref{e.chi_cP_R2}. 
The combination of~\eqref{e.c030502} and~\eqref{e.c030501} finishes the proof.
%
%
%
%The proof is then finished by applying Young's inequality to the estimates~\eqref{e.c030502} and~\eqref{e.c030501} of $I_1$ and $I_2$, respectively.
\end{proof}

\subsubsection{The proof of the Nash-type inequality on $\cN_R$}
\label{s.Nash_proof}

\begin{proof}[Proof of \Cref{p.Nash}]
Our proof proceeds by an interpolation argument using a mollifier.  With this in mind, take any compactly supported, nonnegative, smooth function $\psi$ such that
\be\label{e.c050302}
	\int \psi(z) \dz = 1.
\ee
Up to translation and dilation, we may assume that
\be\label{e.c050301}
	\supp \psi \subset \{z \in (\sfrac12,1)\times \R^{2d} : d_\kin(0,z)\leq 1\}.
\ee
For $\eps\in (0, 1)$ to be chosen and any $s\in (0,\sqrt{\eps R}]$, define
\be\label{e.w04151}
	\begin{split}
		g_s(z)
			& = g * \psi_s(z)
			= \int g(z \circ \tilde z^{-1} ) \psi(\delta_{\sfrac1s}(\tilde z)) \frac{\dtz}{s^{4d+2}}
			= \int g(z\circ \delta_s(\tilde z)^{-1}) \psi(\tilde z) \dtz,
%%			\\&
%%			= \int \overline g(s, w, y) \psi\left( \frac{t-s}{(\sqrt{\eps R})^2}, \frac{x- y - (t-s) w}{(\sqrt{\eps R})^3}, \frac{v-w}{(\sqrt{\eps R})} \right) \frac{\dw\dy\ds}{(\sqrt{\eps R})^6}
%%			\\&
%%			= \int \overline g(t-(\sqrt{\eps R}) s^2, x - (\sqrt{\eps R})^3 y - (\sqrt{\eps R})^2 s(v-(\sqrt{\eps R}) w), v-(\sqrt{\eps R}) w) \psi(s,y,w) \dw\dy\ds,
%			\\&
%			=\int g(\tilde z) \psi\left( \delta_{\sfrac1s}(z \circ \tilde z^{-1}) \right) \frac{\dtz}{s^{4d+2}}
%,
	\end{split}
\ee
where, for all $z$,
\be
	\psi_s(z) := \frac{1}{s^{4d+2}}\psi\left( \delta_{\sfrac1s}(z)\right)
	\qquad\text{ satisfies }
	\quad
	\int \psi_s \dw\dy\ds = 1
\ee
by~\eqref{e.c050302} and a standard change of variables.  Clearly, $g_s \to g$ as $s\to0$.

We note that the order of convolution does not matter here; in our arguments, only the scaling plays an important role.  Indeed, one could argue similarly using $g_s = \psi_s * g$ instead.

For later, we note that, recalling the definition in~\eqref{e.defQ},
\be\label{e.c050401}
	(\supp \psi_s)^{-1}
		\subset \{ z \in [0,s] \times \R^{2d} : d_\kin(0,z) \leq s\}^{-1}
		= \overline Q_s.
\ee

In order to localize $g$ to the domain $\cN_R$, we use a cutoff function $\chi_{\cN_R}$ such that
\be\label{e.chi_cN_R1}
	\chi_{\cN_R} \equiv 1 \qquad \text{ in } \cN_R
	\quad\text{ and }\quad
	\chi_{\cN_R} \equiv 0 \qquad \text{ in } \cN_{\sfrac{3R}{4}}^c,
\ee
while
\be\label{e.chi_cN_R2}
	\|\nabla_v \chi_{\cN_R}\|_{L^\infty} \lesssim \frac{1}{\sqrt R}
	\quad\text{ and }\quad
	\|\nabla_x \chi_{\cN_R}\|_{L^\infty} \lesssim \frac{1}{R^{\sfrac32}}.
\ee
This can be constructed exactly as in~\eqref{e.chi_cP_R1}-\eqref{e.chi_cP_R3} for $\chi_{\cP_R}$.

Before embarking on the estimate, let us understand the supports of the various functions.  First, clearly, up to decreasing $\eps$,
\be
	\supp\left(\chi_R\, g_{\sqrt{\eps R}}\right)
		\subset \cN_{\sfrac{3R}{4}}
\ee
and
\be
	\int_{T_1}^{T_2} \int_{\cN_R} g^2 \dz
		\leq \int_{\Omega_R} \chi_R\,  g^2 \dz
\ee
due to \eqref{e.chi_cN_R1}-\eqref{e.chi_cN_R2}.  Here we have made the change of notation to
\be\label{e.c050707}
	\Omega_R := [T_1 - \eps R, T_2]\times \cN_{\sfrac{R}{2}}
\ee
for simplicity.  Hence, we have
\be\label{e.w06301}
	\int_{T_1}^{T_2} \int_{\cN_R} g^2 \dz
		\leq \int_{T_1}^{T_2} \int \chi_R g_{\sqrt{\eps R}}^2 \dz
			+ \int_{T_1}^{T_2} \int \chi_R (g^2 - g_{\sqrt{\eps R}}^2) \dz
		=: I_1 + I_2.
\ee
Our goal is to show that
\be\label{e.c050303}
	I_1 \lesssim  \frac{1}{(\eps R)^{2d+1}} \|g\|_{L^1(\Omega_R)}^2,
\ee
and
\be\label{e.c050304}
	I_2 \lesssim \sqrt{\eps R} \lbr g\rbr_{H^1_{\rm kin}(\Omega_R)}
			\|g\|_{L^2(\Omega_R)}.
\ee
Indeed, this would complete the proof.

\medskip
\noindent {\bf \# Step one: applying Young's convolution inequality  to $I_1$.}
We begin by analyzing $I_1$, which is the simpler of the two cases.  Here, we simply use the kinetic version of Young's inequality for convolutions: \Cref{l.Youngs}.  Indeed, we have
\be
	I_1
		\leq \|g*\psi_{\sqrt{\eps R}}\|_{L^2([T_1,T_2]\times \cN_{\sfrac{3R}{4}})}^2
		\leq \|g\|_{L^1(([T_1\times T_2]\times \cN_{\sfrac{3R}{4}})\circ Q_{\sqrt{\eps R}})}^2 \|\psi_{\eps R}\|_{L^2}^2,
\ee
where we used~\eqref{e.c050401} to analyze the support of $\psi_{\sqrt{\eps R}}$.  Using only the definition of $N_\cdot$ and the triangle inequality for $d_\kin$, it is easy to see that, up to decreasing $\eps$, we have
\be
	\left(\left([T_1,T_2]\times \cN_{\sfrac{3R}{4}}\right) \circ Q_{\sqrt{\eps R}}\right)
		\subset [T_1 - \eps R, T_2]\times \cN_{\sfrac{R}{2}}
		= \Omega_R,
\ee
as desired.  Additionally, a straightfoward computation yields
\be
	\|\psi_{\sqrt{\eps R}}\|_{L^2}^2
		\lesssim \frac{1}{(\eps R)^{2d + 1}}.
\ee
Hence,~\eqref{e.c050303} is proved.

\medskip
\noindent {\bf \# Step two: rewriting $I_2$ as a series of integrals.}
Let us alter our notation:
\be
	g_{a,b,c}(z)
		:= \int \underbrace{g\big( z \circ (0,0, a(-\tilde v+ \sfrac{\tilde x}{\tilde t})) \circ (-b\tilde t,0,0) \circ (0, 0, \sfrac{-c\tilde x}{\tilde t})\big)}_{=: g_{a,b,c}(z;\tilde z)}
			\psi(\tilde z) \dtz.
\ee
In the sequel, this is useful because the first and third group actions correspond to shifts in $v$, which are represented in the $H^1_\kin$-norm by the $L^2$-norm of $\nabla_v g$, while the second group action corresponds to a shift along transport, which is represented in the $H^1_\kin$-norm by the $L^2_{t,x}H^{-1}_v$-norm of $Yg$.  

We clearly have
\be
	g_{0,0,0}(z) = g(z)
		\quad\text{ and }\quad
	g_{s,s^2,s}(z)
		= g_s(z).
\ee
The fundamental theorem of calculus then yields
\be\label{e.c050708}
	\begin{split}
		g(z)^2 - g_{\sqrt{\eps R}}(z)^2
			= &-2 \int_0^{\sqrt{\eps R}} g_{a,0,0}(z) \partial_a g_{a,0,0}(z)  \,da
				\\&
				- 2\int_0^{\eps R}  g_{\sqrt{\eps R},b,0}(z) \partial_b g_{\sqrt{\eps R},b,0}( z)  \,db
				\\&
				- 2\int_0^{\sqrt{\eps R}} g_{\sqrt{\eps R},\eps R,c}(z) \partial_c g_{\sqrt{\eps R};\eps R,c}(z)   \,dc.
	\end{split}
\ee

Let us write $I_2$, defined in~\eqref{e.w06301}, as
\be
	I_2 = I_{21} + I_{22} + I_{23},
\ee
where each of $I_{2k}$ above corresponds, respectively, to a term in~\eqref{e.c050708}.

\medskip
\noindent
{\bf \# Step three: bound $I_{21}$.}
This case is simple.  It follows by, in turn, directly computing the $a$ derivative in~\eqref{e.c050708}, changing the order of integration, using the Cauchy-Schwarz inequality, and noticing that
\be
	\left| - \tilde v + \frac{\tilde x}{\tilde t}\right| \lesssim 1,
\ee
due to the choice~\eqref{e.c050301} of the support of $\psi$.
Indeed, we find:
\be\label{e.c050709}
	\begin{split}
		I_{21}
			&= -2 \int_{T_1}^{T_2} \int \chi_R(z) \int_0^{\sqrt{\eps R}} \int g_{a,0,0}(z;\tilde z) \partial_a g_{a,0,0}(z;\tilde z) \psi(\tilde z) \dtz \,da \dz
			\\&
			= -2 \int_{T_1}^{T_2} \int \chi_R(z) \int_0^{\sqrt{\eps R}} \int g\left(t,x,v + a\left(-\tilde v + \frac{\tilde x}{\tilde t}\right)\right)
			\\&\hspace{1.85in} \left(-\tilde v + \frac{\tilde x}{\tilde t}\right)\cdot \nabla_v g\left(t,x,v + a\left(-\tilde v + \frac{\tilde x}{\tilde t}\right)\right) \psi(\tilde z) \dtz \,da \dz
			\\&
			= -2\int_0^{\sqrt{\eps R}} \int \int_{T_1}^{T_2} \int \chi_R(x,v + a\tilde v - \frac{a\tilde x}{\tilde t}) g(z) \left(-\tilde v + \frac{\tilde x}{\tilde t}\right) \cdot\nabla_v g(z) \psi(\tilde z) \dz \dtz \, da
			\\&
			\lesssim \int_0^{\sqrt{\eps R}} \int \int_{T_1}^{T_2} \int \|g\|_{L^2(\Omega_R)} \|\nabla_v g\|_{L^2(\Omega_R)} \psi(\tilde z)  \dtz \, da
			\\&
			\leq \sqrt{\eps R}\|g\|_{L^2(\Omega_R)}\lbr g\rbr_{H^1_\kin(\Omega_R)}.
	\end{split}
\ee
In the first inequality, we used that
\be\label{e.c050710}
	[T_1,T_2]\times \supp_{x,v} \chi_R(\cdot,\cdot + a\tilde v - \sfrac{a\tilde x}{\tilde t})
		\subset [T_1-\eps R, T_2]\times \cN_{\sfrac{R}{2}}
		= \Omega_R.
\ee
The inclusion in the time variable is simply due to enlarging the domain.  We show the inclusion of the spatial and velocity variables by using the triangle inequality, the fact that $|a|\leq \sqrt{\eps R}$, and by decreasing $\eps$ if necessary.  Indeed, fix any point
\be
	(x,v) \in \supp \chi_R(\cdot, \cdot + a\tilde v)
\ee
and any point $z_\partial \in \R\times \partial \H \times \R^d$.   Then, by triangle inequality
\be
	\begin{split}
		d_\kin\left(z_\partial, (0,x,v)\right)
			&\geq d_\kin\left( z_\partial, (0,x,v+ a\tilde v - \sfrac{a\tilde x}{\tilde t})\right)
				- 	d_\kin\left((0,x,v+ a\tilde v - \sfrac{a\tilde x}{\tilde t}), (0,x,v)\right)
			\\&
			\geq \sqrt\frac{3R}{4}
				- \VERT (0,-x,-v)\circ (0,x,v+ a\tilde v - \sfrac{a\tilde x}{\tilde t}) \VERT
			\\&
			\geq \sqrt \frac{3R}{4}
				- \VERT (0,0,a\tilde v - \sfrac{a\tilde x}{\tilde t})\VERT
			=
			 \sqrt\frac{3R}{4}
				- a\VERT (0,0,\tilde v - \sfrac{\tilde x}{\tilde t})\VERT.
	\end{split}
\ee
The conclusion then follows by using that $|a| \leq \sqrt{\eps R}$ and the choice~\eqref{e.c050301} of support of $\psi$, which makes $\VERT (0,0,\tilde v - \sfrac{\tilde x}{\tilde t})\VERT \lesssim 1$ uniformly over the support of $\psi$.

A bound of the type~\eqref{e.c050304} then follows from applying Young's inequality to in the last line of~\eqref{e.c050709}.

\medskip
\noindent
{\bf \# Step four: bound $I_{22}$.}
This is the most difficult term as it involves because it requires arguing by the $H^{-1}_v$-$H^1_v$ pairing.  To access this, we begin by directly computing the $b$ derivative in~\eqref{e.c050708}, changing the order of integration, and then changing variables:
\be
	\begin{split}
		I_{22}
			&= 2 \int_{T_1}^{T_2} \int \chi_R(z) \int_0^{\eps R} \int g_{\sqrt{\eps R},b,0}(z;\tilde z) \partial_b g_{\sqrt{\eps R},b,0}(z;\tilde z) \psi(\tilde z) \dtz \,da \dz
			\\&
			= 2\int_0^{\eps R} \int \tilde t \psi(\tilde z) \int_{T_1}^{T_2} \int \chi_R(z)
%			\\&\qquad\qquad\cdot
				g\left[t-b \tilde t,x - b\tilde t\left(v + \sqrt{\eps R}(-\tilde v + \sfrac{\tilde x}{\tilde t})\right) , v + \sqrt{\eps R}(-\tilde v + \sfrac{\tilde x}{\tilde t})\right]
			\\&
				\qquad\qquad\times	\left(Y u\right)\left[t-b\tilde t,x - b\tilde t \left(v + \sqrt{\eps R}(-\tilde v + \sfrac{\tilde x}{\tilde t})\right) , v + \sqrt{\eps R}(-\tilde v + \sfrac{\tilde x}{\tilde t})\right]  \dz \dtz \,da
			\\&
			= 2\int_0^{\eps R} \int \tilde t \psi(\tilde z) \int_{T_1-b\tilde t}^{T_2- b\tilde t} \int \chi_R\left[x + b \tilde t \left(v + \sqrt{\eps R}(-\tilde v + \sfrac{\tilde x}{\tilde t})\right), v + \sqrt{\eps R}(-\tilde v + \sfrac{\tilde x}{\tilde t}) \right]
%			\\&\qquad\qquad\cdot
			\\&
			\hspace{2 in}
			g(z) (Yu)(z) \dz \dtz \, da.
	\end{split}
\ee
We momentarily simplify the notation for the cutoff term $\chi_R$, letting
\be
	\chi_R(x,v;b;\tilde z)
		= \chi_R\left(x +b \tilde t \left(v + \sqrt{\eps R}(-\tilde v + \sfrac{\tilde x}{\tilde t})\right) , v + \sqrt{\eps R}(-\tilde v + \sfrac{\tilde x}{\tilde t})\right).
\ee
Next, arguing as in the justification of~\eqref{e.c050710}, we see that
\be
	[T_1-b \tilde t, T_2 - b \tilde t] \times \supp_{x,v} \chi_R(\cdot,\cdot; b; \tilde z)
%		\subset [T_1 - \eps R, T_2]\times \cN_{\sfrac{R}{2}}
		\subset \Omega_R,
\ee
for any $\tilde z \in \supp \psi$ and any $b\in[0,\eps R]$.  We recall the choice of support of $\psi$ in~\eqref{e.c050301}.
Hence, using the $H^{-1}_v$-$H^1_v$ pairing, we have
\be
	\begin{split}
		I_{22}
			&\lesssim
			\int_0^{\eps R}
			\int \psi(\tilde z)
				\left\|
					\nabla_v\left(
						g\,
						\chi_R(\cdot,\cdot;b;\tilde z)\right)
					\right\|_{L^2(\Omega_R)}
				\lbr g \rbr_{H^1_\kin(\Omega_R)}\dtz \, da.
	\end{split}
\ee
A direct computation using~\eqref{e.chi_cN_R2}, it is easy to see that
\be
	\begin{split}
		\left\|	\nabla_v\left(g\,\chi_R(\cdot,\cdot;b;\tilde z)\right)\right\|_{L^2(\Omega_R)}
			&\lesssim
				\left\|\nabla_v g\right\|_{L^2(\Omega_R)}
				+ \left( \frac{b}{R^{\sfrac32}} + \frac{1}{\sqrt R} \right)\left\|g\right\|_{L^2(\Omega_R)}
			\\&
			\lesssim
			\lbr g \rbr_{H^1_\kin(\Omega_R)}
				+ \left( \frac{b}{R^{\sfrac32}} + \frac{1}{\sqrt R} \right)\left\|g\right\|_{L^2(\Omega_R)}.
	\end{split}
\ee
Hence,
\be
	\begin{split}
		I_{22}
			&\lesssim
			\int_0^{\eps R}
			\int \psi(\tilde z)
				\left(\lbr g \rbr_{H^1_\kin(\Omega_R)}
				+ \left( \frac{b}{R^{\sfrac32}} + \frac{1}{\sqrt R} \right)\left\|g\right\|_{L^2(\Omega_R)}\right)
				\lbr g \rbr_{H^1_\kin(\Omega_R)}\dtz \, db
			\\&
			\lesssim 
				\eps R \lbr g \rbr_{H^1_\kin(\Omega_R)}
				+ \sqrt{\eps R} \left\|g\right\|_{L^2(\Omega_R)}\lbr u \rbr_{H^1_\kin(\Omega_R)}.
	\end{split}
\ee
Again, the proof is then finished after applying Young's inequality.

\medskip
\noindent
{\bf \# Step five: bound $I_{23}$.}
The proof of this is exactly the same as the proof of the bound of $I_{21}$ in step two.  Indeed, this is only a shift in $v$, which is the simplest case.  As such, we omit it.  This concludes the proof of~\eqref{e.c050301} and, thus, \Cref{p.Nash}.
\end{proof}

\subsubsection{The proof of the Poincar\'e-type inequality on $\cO_R$}

While $\cP_R$ and $\cN_R$ are, respectively, increasing and decreasing in $R$, $\cO_R$ is not monotonic in $R$.  This monotonicity was useful in constructing cutoff functions.  In this case, we must define, for any $R'>0$,
\be
	\Theta_{R'} = \{(x,v) \in (\H \times \R^d) \setminus \cP_{\sfrac{R'}{2}} : \dist(\R\times \gamma_+) \leq \sqrt{\sfrac{R'}{10}}\}.
\ee
In the proof of \Cref{p.outgoing}, we use a cutoff function on that is one on $\cO_R$ and zero on $\Theta_{2R}^c$.  A key aspect of the proof is working with the ratio $\sfrac{x_1}{|v_1|}$, so we state a lemma to bound that now.  The proof is postponed to \Cref{s.technical}.
\begin{Lemma}\label{l.cO_R_bounds}
	Fix any $(x,v) \in \Theta_R$.  Then
	\be
		\frac{x_1}{|v_1|}
			\leq \frac{R}{4}.
	\ee
\end{Lemma}

For the proof of \Cref{p.outgoing}, let us make the convention that every norm is taken on $[T_1, T_2+R]\times \H\times \R^d$ unless otherwise specified.  For example, by writing
\be
	\|g\|_{L^2}
		\quad\text{ we mean }\quad
	\|g\|_{L^2([T_1,T_2+R]\times \H\times \R^d)}.
\ee
This saves significant space and does not cost clarity.

\begin{proof}[Proof of \Cref{p.outgoing}]
We extend $g$ to $\overline g$  on $[T_1, T_2 + R]\times\R^d\times(\R_-\times\R^{d-1})$ as follows:  
\be
	\overline g(t,x,v)=
	\begin{cases}
		f(t,x,v) &\text{if $x_1\geq0$,}\\
		f(t-\sfrac{x_1}{v_1},(0,\overline x),v) &\text{if $x_1<0$,}
	\end{cases}
\ee
where $\overline{x}\coloneqq(x_2,\dots,x_d)$.
Notice that
\be\label{eq:ext_int}
	Y\overline f = 0
		\qquad\text{ for }
		x_1 < 0.
\ee

Next, we take a cutoff function $\chi_{\cO_R}$ such that
\be\label{e.chi_cO_R1}
	\begin{split}
		&\chi_{\cO_R} \equiv 1 \quad \text{ in } \cO_R,
		\qquad\text{ and }\qquad
%		\\&
		\chi_{\cO_R} \equiv 0 \quad \text{ in } \Theta_{2R},
			%\left\{(x,v) \in \H \times \R^d \setminus \cP_{\sfrac{R}{2}} : \dist(\R\times \gamma_+,(0,x,v)) \leq \sqrt R\right\},
	\end{split}
\ee
while
\be\label{e.chi_cO_R2}
	\|\nabla_v \chi_{\cO_R}\|_{L^\infty} \lesssim \frac{1}{\sqrt R}
	\quad\text{ and }\quad
	\|\nabla_x \chi_{\cO_R}\|_{L^\infty} \lesssim \frac{1}{R^{\sfrac32}}.
\ee
This can be constructed easily using the scaling properties in~\eqref{e.domain_scaling}; see the discussion around~\eqref{e.chi_cP_R1}-\eqref{e.chi_cP_R2}.  

%Therefore, setting $Y\coloneqq \partial_t + v\cdot \nabla_x$, we have that
%\be \label{eq:ext_int}
%	\int_0^{T} (Y f)(t+r,x+rv,v)  \,dr
%	=	
%	\int_0^{T'} (Y \overline f)(t+r,x+rv,v)  \,dr
%\ee%
%for any $(t,x,v)\in [R,+\infty)\times \cO_R, T<T'$ such that $-T v_1 =x_1$. 
%Indeed, we have that $(Y \overline f)(t+r,x+rv,v)=0$ for any $r\in[T,T']$.

By the fundamental theorem of calculus we have that, for any $z\in [T_1,T_2]\times \Theta_{2R}$, 
\be\label{eq:fund_theorem_O}
\begin{split}
	g(t,x,v)^2
		=& -2\int_0^{\frac{x_1}{|v_1|}} g(t+r,x+rv,v) Yg(t+r,x+rv,v) \dr 
			+ g(t+\sfrac{x_1}{|v_1|},(0,\overline x),v)^2 \\
		=& -2\int_0^{R} \overline g(t+r,x+rv,v) Y\overline{g}(t+r,x+rv,v) \dr
			+ g(t+\sfrac{x_1}{|v_1|},(0,\overline x),v)^2 
\ ,
\end{split}
\ee 
where the second equality follows from \eqref{eq:ext_int} and \Cref{l.cP_R_bounds}. 
Therefore,
\be
\begin{split}
	\int_{T_1}^{T_2} \int_{\cO_R} &g(z)^2 \dz
		\leq	\int_{T_1}^{T_2} \int_{\H\times\R^d} \chi_{\cO_R} g(z)^2 \dz
		\\=&
			-2\int_{T_1}^{T_2} \int_{\H\times\R^d} \int_0^{R} g(t+r,x+rv,v) Y\overline{g}(t+r,x+rv,v) \dr \dx \dv \dt\\
			&
			+\int_{T_1}^{T_2} \int_{\H\times\R^d} \chi_{\cO_R}(x,v) f(t+\sfrac{x_1}{|v_1|},(0,\overline x),v)^2  \dx \dv \dt
			= I_1 + I_2 ,
\end{split}
\ee
where we passed from the first to the second line rewriting $f(t,x,v)^2$ according to \eqref{eq:fund_theorem_O}. 

We first see, by a change of variables, that 
\be
\begin{split}
	I_1
%^		=& -2 \int_0^{R} \int_{T_1}^{T_2}  \int_{\H\times\R^d} \chi_{\cO_R}(x,v)  \overline f(t+r,x+rv,v) Y\overline{f}(t+r,x+rv,v)  \dx \dv \dt \dr
		=& -2\int_0^{R} \int_{T_1+r}^{T_2+r}  \int_{\R^d} \int_{\R_- \times \R^{d-1}} \chi_{\cO_R}(x-rv,v)  \overline g(t,x,v) Y\overline{g}(t,x,v)  \dx \dv \dt \dr
		\\=&
		-2\int_0^{R} \int_{T_1+r}^{T_2+r}  \int_{\H\times\R^d} \chi_{\cO_R}(x-rv,v)  g(t,x,v) Yg(t,x,v) \dx \dv \dt \dr .
\end{split}
\ee
In the second equality, we used~\eqref{eq:ext_int} again to reduce the domain of the integral.

Let $\bar\chi_{\cO_R}(r,x,v) = \chi_{\cO_R}(x-rv,v)$. 
Then, by the $H^{-1}_v$-$H^{1}_v$ pairing, we get
\be
\begin{split}
	|I_1|
		\lesssim&  \int_0^{R}
			\|\nabla_v\(\bar\chi_{\cO_R} g\)\|_{L^2([T_1+r,T_2+r]\times \H \times \R^d)}
			\lbr g\rbr_{H^1_\kin([T_1+r,T_2+r]\times \H \times \R^d)} dr\\
	\lesssim& \lbr f\rbr_{H^1_\kin} 
		\int_0^{R} \bigg(\(\frac{1}{\sqrt{R}}+ \frac{r}{R^{\sfrac 3 2}}\)\|g\|_{L^2}
		+\|\chi_{\cO_R} \nabla g\|_{L^2}\bigg)dr\\
	\lesssim& \sqrt{R} \lbr g\rbr_{H^1_\kin}\|f\|_{L^2}+R\lbr g\rbr_{H^1_\kin}^2.
\end{split}
\ee
In the second line follows from estimates \eqref{e.chi_cO_R2}.

We now estimate $I_2$. Changing variables yields
\be
\begin{split}
I_2 
%& \int_\R \int_{\R_+} \chi_{\cO_R}(x,v) \int_{T_1}^{T_2} f(t+\sfrac{x}{|v|},\overline{x},v)^2 \dt \dx \dv \\ 
%%	=&  \int_{-\infty}^{-1} \int_{0}^{R|v|} \int_{T_1}^{T_2} \chi_{\cO_R}(x,v) f(t+{x}/{|v|},\overline{x},v)^2 \dt \dx \dv  
%\intertext{(by change of variables)}
	=& \int_{\H\times\R^d} \chi_{\cO_R}(x,v) \int_{T_1+\sfrac{x_1}{|v_1|}}^{T_2+\sfrac{x_1}{|v_1|}} g(t,(0,\overline x),v)^2 \dt \dx \dv 
% 	=&  \int_{-\infty}^{-1} \int_{0}^{R|v|} \int_{T_1-{x}/{|v|}}^{T_2-{x}/{|v|}}  f(s,\overline{x},v)^2 ds \ dx \ dv 
%\intertext{(being the integrand positive)}
	\\\leq&
		\int_{\H\times\R^d} \chi_{\cO_R}(x,v) \int_{T_1}^{T_2+R} g(t,(0,\overline x),v)^2 \dt \dx \dv
	\\=&
		\int_{T_1}^{T_2+R}  \int_{\R^d} \int_{\R^{d-1}} \chi_{\cO_R}(x,v)x_1 g(t,(0,\overline x),v)^2 \, d\overline x \dv \dt,
\end{split}
\ee
where we used \Cref{l.cO_R_bounds} and the fact that the integrand is positive to obtain the second line. 

Applying once again \Cref{l.cO_R_bounds}, we find
\be
	I_2
		\leq \frac{R}{4} \int_{T_1}^{T_2+R}  \int_{\R^d} \int_{\R_- \times \R^{d-1}}  |v_1| g(t,(0,\overline x),v)^2 d\overline x \dv \dt.
\ee
This concludes the proof. %Since $\supp \chi_{\cO_R} \subseteq \tilde{O}_{R,2}$, we have that $v_1\in(\-\infty,2\sqrt{R})$ and $x_1\in(0,R|v|)$ (see Remark \ref{rem:tildeO}) 
%and then
%\be
%\begin{split}
%|I_2| 	\leq& \int_{(-\infty,2\sqrt{R}]\times\R^{d-1}} \int_{[0,2R|v_1|]\times\R^{d-1}} \int_{T_1}^{T_2+R} f(t,\overline{x},v)^2 \dt \dx \dv \\
% 	\leq& R \int_{T_1}^{T_2+R} \int_{\R_-\times\R^{d-1}} \int_{\R^{d-1}}|v_1| f(t,\overline{x},v)^2  d\overline{x} \ dv \ dt
%	\ ,
%\end{split}
%\ee
%using the abuse of notation $d\overline{x}=d(x_2,\dots,x_d)$ and the fact that $f(t,x,v)=0$ if $x_1=0$ and $v_1>0$.
%
%Then, combining the estimates for $I_1$ and $I_2$ (and applying Young inequality) we get
%\be
%\begin{split}
%	\int_{T_1}^{T_2} \int_{\cO_R} f(t,x,v)^2 dx \ dv \ dt
%\leq	&R\lbr f\rbr_{H^1_\kin([T_1,T_2+R]\times \H \times \R^d)}^2  \\
%	&+  R \int_{T_1}^{T_2+R} \int_{\R_-\times\R^{d-1}} \int_{\R^{d-1}}|v_1| f(t,\overline{x},v)^2  d\overline{x} \ dv \ dt
%	\ .
%\end{split}
%\ee
\end{proof}

\section{Controlling \texorpdfstring{$f$}{f} on the isolated region: \texorpdfstring{\Cref{l.isolated_estimate}}{Lemma \ref{l.isolated_esimate}}}\label{s.zeta}

We begin with a lemma that is simple to prove.  It essentially says if a point $(x,v)$ is distance $\rho$ to the boundary, then the path from $(x,v)$ to the boundary that simply follows transport (without any changes in velocity) has to take at least time $\rho$.
\begin{Lemma}\label{l.transport}
	If $(x,v) \in \cN_R$, then
	\be
		x_1
			\geq \frac{R}{10} |v_1|.
%		\frac{x_1}{|v_1|} \geq \frac{R}{10}.
	\ee
\end{Lemma}
\begin{proof}
	This is trivially true if $v_1 = 0$, so we assume that $v_1 \neq 0$.  Recall from~\eqref{e.regions}, that
	\be
		\dist(\R \times \partial \H \times \R^d, (0,x,v)) \geq \sqrt{\sfrac{R}{10}}.
	\ee
	In view of~\eqref{e.dist2}, it follows that
	\be
		\sqrt{\tfrac{R}{10}}
			\leq \VERT \zeta \VERT
		\qquad\text{ where }
		\zeta
			= (\tau, \xi, \omega)
			= \left( \frac{x_1}{v_1}, (-x_1 - \tau v_1, 0), 0\right).
	\ee
	Unpacking the definition of $\VERT\cdot\VERT$ with the choice $w=0$, we deduce that
	\be
		\Big(\frac{R}{10}\Big)^{\sfrac12}
			\leq |\tau|^{\sfrac12},
	\ee
	which is precisely the claim.
\end{proof}

The proof of \Cref{l.isolated_estimate} is fairly straightforward, if tedious.
\begin{proof}[Proof of \Cref{l.isolated_estimate}]
To simplify the notation in this proof, let us set
\be
	\overline R = \frac{R}{10}.
\ee
For $\psi$ and $E$ to be chosen, we define
\be\label{e.zeta}
	\tilde \mu_R(x,v)
		= E \left( - \frac{\overline Rv_1}{x_1} \right)
			\psi\left( - \frac{v_1^3}{x_1}\right)
			\psi\left( - \frac{2v_1}{\sqrt{\overline R}} - 1\right).
\ee
Below, it will be helpful to suppress the arguments, while keeping track of the three individual functions that make up $\tilde \mu_R$.  To this end, we write
\be
	\tilde \mu_R(x,v)
		= E \psi \hat \psi,
\ee
where $\psi$ is shorthand for $\psi(-\sfrac{v_1^3}{x_1})$ and $\hat \psi$ is shorthand for $\psi(-\sfrac{2v_1}{\sqrt {\overline R}} - 1)$.

To aid the reader, let us  stress that, in all nontrivial cases in this proof, $v_1 < 0$.  Thus, $-v_1$, $-v_1^3$, etc. are {\em positive} quantities.

Take $E$ to be a decreasing function
\be\label{e.c041902}
	E(\zz) = 1 \quad\text{ if } \zz\leq 1,
	\qquad\tand\qquad
	E(\zz) \approx e^{-\zz}
		\quad\text{ if } \zz\geq 0
\ee
such that
\be\label{e.c041901}
	E'' + E' \leq 0
		\qquad\text{ for all } \zz \geq 0.
\ee
Moreover, we may take $E$ such that
\be
	|E''|, |E'| \lesssim E.
\ee
Roughly, $E$ is a mollification of $\min\{1,e^{1-\zz}\}$. 
This is somewhat simple to construct, so we omit its proof. 
%Indeed, this holds pointwise apart from $z=1,2$.  Near $z=1$, it is decreasing and concave, whence~\eqref{e.c041901} holds.  Near $z=2$, $E$ is $C^1$, so that there is no $\delta$ component of $E''$ and we can ignore this measure zero part.  If we want it to be smooth, we can convolve it...
Additionally, we let $\psi$ be any increasing function such that
\be
	\psi(\zz)
		= \begin{cases}
			1 \qquad
				& \text{ if } \zz \geq 1,\\
			0 \qquad
				& \text{ if } \zz \leq \sfrac1{2}.
		\end{cases}
\ee

Let us first check that, with these choices, $\tilde \mu_R$ satisfies~\eqref{e.c061804}.  This is straightforward, except for the first case when $(x,v) \in \cN_R \cap \{x_1 \leq - v_1^3\}$.  Clearly $\psi$ satisfies the correct bound.  From \Cref{l.transport}, we have that $x_1 \leq \overline R |v_1|$.  This implies that
\be
	E\left( - \sfrac{v_1 \overline R}{x_1}\right)
		\geq E(1)
		\approx 1.
\ee
Finally, we notice that
\be\label{e.c061806}
	\overline R \leq \frac{x_1}{|v_1|}
		\leq \frac{|v_1|^3}{|v_1|}
		= |v_1|^2.
\ee
It follows that $\hat \psi = 1$.

We now need to show the main estimate~\eqref{e.c061805}. This proceeds by considering each of the subdomains on which $\mu_R$ is nonzero one at a time.  Given its definition~\eqref{e.zeta}, there are eight cases to check: there are three functions, each having one region where it takes the constant value one and one region where it varies.
%  Let us note that we do not need to check either of the cases
%\be
%	\frac{2|v|}{\sqrt R} -1 \leq \frac{1}{2}
%		\quad\text{ or }\quad
%	\frac{|v|^3}{x} \leq \frac{1}{2}.
%\ee
%because here $\chi \equiv 0$.  Additionally, we already understand the case
%\be
%	\frac{|v|R}{x} \leq 1,
%		\quad
%	\frac{2|v|}{\sqrt R} -1 \geq 1,
%		\quad\text{ and }\quad
%	\frac{|v|^3}{x} \geq 1,
%\ee
%because then $\chi \equiv 1$.  
%We now check the remaining nine cases.

\medskip
\noindent{\bf \# Case one:}
		\be
			\frac{-\overline R v_1}{x_1}
				\geq 1,
			\quad
			\frac{-v_1^3}{x_1}
				\geq 1,
			\quad\text{ and } \quad
			\frac{2 v_1}{\sqrt{\overline R}} - 1 \geq 1.
		\ee
		Let us note that the last inequality yields
		\be\label{e.c061807}
			v_1^2 \geq \overline R.
		\ee
		In this case, letting $\zz = \tfrac{-\overline R v_1}{x_1}$,
		\be
			\begin{split}
				\Delta_v \chi + v\cdot\nabla_x \tilde \mu_R
					& = \frac{\overline R^2}{ x_1^2} E'' + \frac{\overline Rv_1^2}{x_1^2} E'
					= \frac{z^2}{v_1^2} \left(E'' + \frac{v_1^2}{\overline R} E'\right)
					\\&
					\leq \frac{z^2}{v_1^2} \left( E'' +  E'\right)
					\leq 0.
%					\\&= \frac{R^2 \alpha}{x^2} E \left( \alpha - \frac{v^2}{R}\right)
%					\leq 0,
			\end{split}
		\ee
		The second-to-last inequality follows by~\eqref{e.c061807} and the fact that $E' \leq 0$. The last inequality follows by~\eqref{e.c041901}. 
%		as long as we choose
%		\be
%			\alpha \leq 1.
%		\ee
		This is clearly~\eqref{e.c061805} in this case.

\medskip		
\noindent{\bf \# Case two:}
		\be\label{e.c061901}
			\frac{-\overline R v_1 }{x_1}
				\geq 1,
			\quad
			\frac{- v_1^3}{x_1}
				\geq 1,
			\quad\text{ and } \quad
			\frac{- 2v_1}{\sqrt {\overline R}} - 1 \in (\sfrac12, 1).
		\ee
		In this case,
		\be\label{e.c040302}
			\begin{split}
				\left(\Delta_v  + v\cdot\nabla_x \right) \tilde \mu_R
					&= \left(\frac{\overline R^2}{x_1^2} E'' \hat \psi + \frac{4 \sqrt{\overline R}}{x_1} E' \hat \psi' + \frac{4}{\overline R} E \hat \psi''\right)
						+ \left(\frac{v_1^2\overline R}{x_1^2} E' \hat \psi\right)
					\\&
					\leq \frac{\overline R^2}{x_1^2} E'' \hat \psi +  \frac{4}{\overline R} E \hat \psi''.
			\end{split}
		\ee
		In the inequality, we used that $E$ is decreasing, while $\hat \psi$ is increasing.
		
		Notice that for any $\eps >0$ and $z\geq 0$,
		\be
			1 \lesssim \frac{x_1}{|v_1|^3} e^{- \frac{\eps v_1^3}{x_1}}
			\quad\text{ and }\quad
			|E'(\zz)|, |E''(\zz)|, E(\zz) \lesssim e^{-\zz}.
		\ee
		We look only at the first term in~\eqref{e.c040302}; however, the second term is handled similarly.  Then,
		\be
			\begin{split}
				\frac{R^2}{x_1^2} E'' \hat \psi
					&\lesssim \frac{\overline R^2}{x_1^2} e^{\frac{v_1 \overline R}{x_1}}
					\lesssim \frac{\overline R^2}{x_1^2} e^{ \frac{v_1 \overline R}{x_1}} \Big( \frac{x_1^3}{|v_1|^9}e^{-\frac{\eps v_1^3}{x_1}}\Big)
					\\&
					= \frac{\overline R^{\sfrac{13}{4}}}{\overline R^{\sfrac54}} \frac{x_1}{|v_1|^9} e^{\frac{v_1\overline R}{x_1}} e^{-\frac{\eps v_1^3}{x_1}}
					\leq \frac{\overline R^{\sfrac{13}{4}}}{\overline R^{\sfrac54}} \frac{x_1}{|v_1|^9} e^{ \frac{v_1^3}{x_1}} e^{-\frac{\eps v_1^3}{x_1}}
					\\&
					\approx
					\frac{\overline R^{\sfrac{13}{4}}}{\overline R^{\sfrac54}} \frac{x_1}{\overline R^{\sfrac{13}{4}} |v_1|^{\sfrac52}} e^{ \frac{v_1^3}{x_1}} e^{-\frac{\eps v_1^3}{x_1}}
					\lesssim \frac{1}{\overline R^{\sfrac54}} \tilde \phi(x,v).
%					
%					
%					
%					\approx \frac{R^{\sfrac{13}4}}{R^{\sfrac54}} \frac{x}{|v|^9}  e^{- (\alpha-\eps) \frac{|v| R}{x}}
%					\\&
%					\lesssim \frac{|v|^{\sfrac{13}{2}}}{R^{\sfrac54}} \frac{x}{ |v|^{\sfrac{18}{2}}} e^{- (\alpha-\eps) \frac{|v|^3}{x}}
%					\lesssim \frac{1}{R^{\sfrac54}} \frac{x}{ |v|^{\sfrac52}} e^{- \frac{|v|^3}{9x}}
%					\approx \frac{1}{R^{\sfrac54}} \phi^*(x,v).
			\end{split}
		\ee
		In the  inequality on the second line and in ``$\approx$'' on the last line, we used that, by the third item in~\eqref{e.c061901},
		\be
			\frac{3\sqrt{\overline R}}{4}
				\leq -v_1
				\leq \sqrt{\overline R}.
		\ee
		This is clearly~\eqref{e.c061805} in this case.
		
\medskip
\noindent{\bf \# Case three:}
		\be
			- \frac{v_1 \overline R}{x_1} \geq 1,
				\quad
			-\frac{v_1^3}{x_1} \in (\sfrac12, 1)
				\quad\text{ and } \quad
			-\frac{2 v_1}{\sqrt{\overline R}} - 1 \geq 1.
		\ee
		We claim that this case cannot happen.  Indeed, the first and third inequalities above yield
		\be
			x_1 \leq -v_1 \overline R
				\leq -v_1^3,
		\ee
		while the second implies that
		\be
			x_1 > - v_1^3.
		\ee
		This is a contradiction.  Hence, case three cannot occur.

\medskip			
\noindent{\bf \# Case four:}
		\be
			-\frac{v_1 \overline R}{x_1}
				\geq 1,
			\quad
			-\frac{v_1^3}{x_1} \in (\sfrac12, 1),
			\quad\text{ and } \quad
			-\frac{2v_1}{\sqrt{\overline R}} - 1 \in (\sfrac12, 1).
		\ee
		This case involves the most derivatives since all three cutoff functions are varying.  That said, it is fundamentally the same as case two, while being slightly easier because
		\be
			x_1 \approx -v_1^3 \approx \overline R^{\sfrac32}
			\quad\text{ and }\quad
			\tilde \phi(x,v) \approx \sqrt{-v_1}.
		\ee
		As such, we omit its proof.

\medskip		
\noindent{\bf \# Case five:}
		\be
			- \frac{v_1 \overline R}{x_1}
				< 1,
			\quad
			-\frac{v_1^3}{x_1} \geq 1,
				\quad\text{ and } \quad
			-\frac{2v_1}{\sqrt{\overline R}} - 1 \geq 1.
		\ee
		This case is precisely when $\tilde \mu_R \equiv 1$.  Hence
		\be
			(\Delta_v + v\cdot\nabla_x) \tilde \mu_R = 0,
		\ee
		which clearly yields~\eqref{e.c061805} in this case.

\medskip		
\noindent{\bf \# Case six:}
		\be\label{e.c040304}
			- \frac{v_1 \overline R}{x_1}
				< 1,
			\quad
			-\frac{v_1^3}{x_1}
				\geq 1
			\quad\text{ and } \quad
			- \frac{2v_1}{\sqrt{\overline R}} - 1 \in (\sfrac12, 1).
		\ee
		This case cannot occur.  Indeed, the first and second inequalities in~\eqref{e.c040304} imply that
		\be
			-v_1 \overline R < x_1 \leq - v_1^3,
		\ee
		which implies that $v_1^2 > R$.  On the other hand, the last inequality in~\eqref{e.c040304} implies that
		\be
			-v_1 < \sqrt{\overline R}.
		\ee
		These are in contradiction (recall that $-v_1 \geq 0$).
				
\medskip		
\noindent{\bf \# Case seven:}
		\be
			-\frac{v_1 \overline R}{x_1}
				<  1,
			\quad
			-\frac{v_1^3}{x_1} \in (\sfrac12, 1)
				\quad\text{ and } \quad
			-\frac{2 v_1}{\sqrt{\overline R}} - 1 \geq 1.
		\ee
		From the first and second inequalities, we see that
		\be
			x_1 \approx -v_1^3 \gtrsim \overline R^{\sfrac32}.
		\ee
		Hence,
		\be
			\tilde \phi(x,v)
				\approx \sqrt{-v_1}.
		\ee
		We also notice that $\psi$ is the only term in $\tilde \mu_R$ not equal to $1$ on this domain.  Hence,
		\be
			\begin{split}
				\left(\Delta_v + v\cdot\nabla_x\right) \tilde \mu_R
					&= \frac{9 v^4}{x^2} \psi'' - \frac{6v}{x} \psi'
						+ \frac{v^4}{x^2} \psi'
					\\&
					\lesssim \frac{1}{v_1^2}
					\lesssim \frac{\sqrt{-v_1}}{ \overline R^{\sfrac54}}
					\approx \frac{\tilde\phi(x,v)}{\overline R^{\sfrac54}}.
			\end{split}
		\ee
		Thus, we have established~\eqref{e.c061805}.

\medskip		
\noindent{\bf \# Case eight:}
		\be
			-\frac{v_1 \overline R}{x_1}
				< 1,
			\quad
			-\frac{v_1^3}{x_1}
				\in (\sfrac12, 1),
			\quad\text{ and } \quad
			\frac{2v_1}{\sqrt{\overline R}} - 1
				\in (\sfrac12, 1).
		\ee
		In this case, we have
		\be
			x_1 \approx -v_1^3 \approx \overline R^{\sfrac32}.
		\ee
		From this point, the proof is essentially the same as in the previous case, and, thus, is skipped.  This completes the proof of \Cref{l.isolated_estimate}.
\end{proof}

\section{Other technical lemmas}\label{s.technical}

\subsection{Understanding \texorpdfstring{$(x,v)$}{(x,v)} in \texorpdfstring{$\cP_R$}{P}}

\begin{proof}[Proof of \texorpdfstring{\Cref{l.cP_R_bounds}}{Lemma \ref{l.cP_R_bounds}}]
	Let $z = (0,x,v)$ for ease.  It is easy to see that the infimum in~\eqref{e.dist2} is attained, up to including the boundary $(\{0\}\times \R^{d-1})^2$, so we fix $\zeta = (\tau,\xi,\omega)$ such that
	\be\label{e.c041201}
		\begin{split}
			&z \circ \zeta \in \R \times \overline\gamma_-
				\quad\tand
			\\&
			\VERT \zeta\VERT
				= \dist(\R \times \gamma_-,z)
				\leq \sqrt R.
		\end{split}
	\ee
	Since $z \circ \zeta = (\tau, x + \xi + \tau v, v + \omega)$, then we have the constraints
	\be\label{e.c041202}
		0= x_1 + \xi_1 + \tau v_1
		\quad\tand\quad
		v_1 + \omega_1 >0.
	\ee
	It is clear from the second inequality in~\eqref{e.c041201}, as well as the definition~\eqref{e.norm} of $\VERT\cdot\VERT$, that
	\be
		|\omega_1|\leq 2\sqrt R,
	\ee
	This can be seen by noting that
	\be\label{e.c041204}
		\sqrt R
			\geq \min_w \max\{|\omega-w|, |w|\}
			\geq \frac{|\omega|}{2}.
	\ee
	Hence, the second line of~\eqref{e.c041202} yields
	\be
		- v_1 < \omega_1 \leq 2\sqrt R,
	\ee
	as desired.

	Next, notice that the definition~\eqref{e.norm} of $\VERT \cdot\VERT$ and the second line of~\eqref{e.c041201} implies that $|\tau| \leq R$.  Hence, if
	\be\label{e.c041203}
		|\xi_1|\leq 3 R^{\sfrac32}
	\ee
	then we immediately have
	\be
		x_1
			= - \xi_1 - \tau v_1
			\lesssim \max\{ 3R^{\sfrac32}, R v_1\},
	\ee
	as desired.
	
	To see~\eqref{e.c041203}, let $w$ be the minimizer in the definition~\eqref{e.norm} of $\VERT \cdot\VERT$.  Then, by arguing exactly as in~\eqref{e.c041204}, we find $|w| \leq 2\sqrt R$.  We deduce that
	\be
		R^{\sfrac32}
			\geq |\xi - \tau w|
			\geq |\xi| - 2R^{\sfrac32},
	\ee
	from which~\eqref{e.c041203} follows.  This concludes the proof.
\end{proof}

\subsection{Understanding \texorpdfstring{$(x,v)$}{(x,v)} in \texorpdfstring{$\cO_R$}{O}}

\begin{proof}[Proof of \Cref{l.cO_R_bounds}]
By the symmetry of $\gamma_-$ and $\gamma_+$, we immediately see that
\be\label{e.c061802}
	x_1 \leq \frac{R}{10} \max\{ -v_1, 3 \sqrt{\sfrac{R}{10}}\}
\ee
since $\dist(\R\times\gamma_-, (0,x,v)) \leq \sqrt{\sfrac{R}{10}}$; see  \Cref{l.cP_R_bounds}.  This is useful in the sequel.

If $-v_1 \geq \sqrt{\sfrac{R}{2}}$, then we find
\be
	\frac{x_1}{|v_1|}
		\leq \begin{cases}
				\frac{3R}{10^{\sfrac{3}{2}}\sqrt{\sfrac{1}{2}}}
					\qquad &\text{ if } -v_1 \leq 3 \sqrt{\sfrac{R}{10}},\\
				\frac{R}{10}
					\qquad &\text{ if } -v_1 \geq 3 \sqrt{\sfrac{R}{10}}.
			\end{cases}
\ee
Hence, 
\be
	\frac{x_1}{|v_1|}
		\leq \frac{R}{5},
\ee
and the proof is finished in this case.

Next consider when
\be\label{e.c061801}
	|v_1| = -v_1 < \sqrt{\sfrac{R}{2}}.
\ee
In view of~\eqref{e.c061802}, this yields
\be\label{e.c061803}
	x_1
		\leq \frac{3R^{\sfrac32}}{10^{\sfrac32}}
		< \frac{R^{\sfrac32}}{2^{\sfrac32}}.
\ee
By definition, we find
\be
	\dist( \R \times \gamma_-, (0,x,v)) \geq \sqrt{\sfrac{R}{2}}.
\ee
Letting
\be
	\zeta
		= \left(\tfrac{R}{4},
				\left(- x_1 - \tfrac{R}{4}v_1,0\),
				\(-v_1,0\)
			\right)
\ee
we have that $z \circ \zeta \in \overline {\R \times \gamma_-}$ and, hence,
\be
	\VERT \zeta \VERT \geq \sqrt{\sfrac{R}{2}}.
\ee
Taking $w = 0$ in the definition~\eqref{e.norm} of $\VERT\cdot\VERT$, we find
\be
	\max\left\{
		|\tfrac{R}{4}|^{\sfrac12},
		\left|x_1 + \tfrac{Rv_1}{4}\right|^{\sfrac13},
		|0|,
		|v_1-0|
	\right\}
		\geq \sqrt{\sfrac{R}{2}}.
\ee
By~\eqref{e.c061801}, it follows that
\be
	\Big(\frac{R}{2}\Big)^{\sfrac32}
		\leq \left|x_1 + \frac{Rv_1}{4}\right|.
\ee
If the term in the absolute value is negative, we find
\be
	x_1 < - \frac{Rv_1}{4}
		= \frac{R|v_1|}{4}
\ee
from which the conclusion follows.  If the term in absolute value is nonnegative, we find
\be
	\Big(\frac{R}{2}\Big)^{\sfrac32} + \frac{R|v_1|}{4}
		\leq x_1.
\ee
This contradicts~\eqref{e.c061803}.  The proof is concluded.
\end{proof}

\subsection{Understanding \texorpdfstring{$(x,v)$}{(x,v)} in \texorpdfstring{$\cN_R$}{N}}

\begin{proof}[Proof of \Cref{l.cN_R}]
	Fix $\eps>0$ to be chosen.  Let us first consider the case where $|v_1| \geq \eps \sqrt R$.  
	Applying \Cref{l.transport}, we see that
	\be\label{e.c061909}
		x_1 \geq \frac{R |v_1|}{10}
			\geq \frac{\eps R^{\sfrac32}}{10}.
	\ee
	If $x_1 \geq |v_1|^3$ then, by~\eqref{e.phi_asymp},
	\be
		\tilde \phi(x,v)
			\approx x_1^{\sfrac16}
			\gtrsim R^{\sfrac14}.
	\ee
	We used~\eqref{e.c061909} in the last inequality. 
	
	If $x_1 < |v_1|^3$, then, applying the assumption $v_1 >0$ and, by~\eqref{e.phi_asymp},
	\be
		\tilde \phi(x,v)
			\approx \sqrt{ v_1}
			\gtrsim R^{\sfrac14}. 
	\ee
	This finishes the proof in this case.
	
	Now we consider the case $|v_1| \leq \eps\sqrt R$.  Recall from~\eqref{e.regions}, that
	\be
		\dist( \R_+ \times \partial \H \times \R^d, (0,x,v))
			\geq \sqrt{\tfrac{R}{10}}.
	\ee
	Let
	\be
		\zeta = (\tau, \xi, 0)
			= (\sfrac{R}{20}, (- x_1 - \tau v_1,0), 0),
	\ee
	and notice that $0 = x_1 + \xi_1 + \tau v_1$.  Hence,
	\be
		\sqrt{\tfrac{R}{10}}
			\leq \dist( \R_+ \times \partial \H \times \R^d, (0,x,v))
			\leq \VERT \zeta\VERT.
	\ee
	Taking $w=0$ in the definition~\eqref{e.norm} of $\VERT\cdot\VERT$, we find
	\be
		\sqrt{\tfrac{R}{10}}
			\leq \max \left\{ |\sfrac{R}{20}|^{\sfrac12}, |x_1 + \tau v_1|^{\sfrac13}, |0|,|0| \right\}.
	\ee
	It follows that
	\be
		\sqrt{\tfrac{R}{10}}
			\leq |x_1 + \tau v_1|^{\sfrac13}.
	\ee
	Rearranging this, we find
	\be
		\frac{R^3}{10^3} - \frac{R}{20} |v_1|
			\leq x_1.
	\ee
	Using that $|v_1| \leq \eps \sqrt R$ and possibly decreasing $\eps$, we deduce that
	\be
		\frac{R^3}{10^4}
			\leq x_1.
	\ee
	Further decreasing $\eps$, if necessary, we see that $|v_1|^2 \leq x_1$, whence
	\be
		\tilde \phi(x,v) \approx x_1^{\sfrac16}
			\gtrsim R^{\sfrac14}.
	\ee
	Here, we once again used the asymptotics of $\tilde \phi$ given in~\eqref{e.phi_asymp}. 
	This concludes the proof.	
\end{proof}

\section{The whole space case: \texorpdfstring{\Cref{c.ws_time_decay}}{Corollary \ref{c.ws_time_decay}}}\label{s.ws_time_decay}

The proof of \Cref{t.ws_Nash} follows exactly the outline of the proof of \Cref{p.Nash}.  The only modification to be made is to take the cutoff function $\psi$ in~\eqref{e.c050302}-\eqref{e.c050301} to be supported on $B$.  As such, we omit the proof.

In this section, we provide a brief outline of the proof of the whole-space time decay.  The work here is similar to, but much simpler than, the proof of \Cref{t.main}.
\begin{proof}[Proof of \Cref{c.ws_time_decay}]
	Let us note that
	\be
		\frac{d}{dt} \int f \dx\dv
			= \int (\nabla\cdot(a\nabla_v f) - v\cdot\nabla_x f) \dx\dv
			= 0.
	\ee
	Hence, the $L^1$-norm is conserved:
	\be\label{e.c062004}
		\|f(t,\cdot,\cdot)\|_{L^1(\R^{2d})} = \int f_\init \dx\dv.
	\ee

	The main step is obtaining an $L^1 \to L^2$ bound on the solution operator $S_t f_\init = f(t)$.  We do this by combining the Nash inequality with the energy equality.
		
	Let us begin with the energy equality.  Multiplying~\eqref{e.ws_kfp} by $f$, integrating, and then integrating by parts, we obtain, for any $0 \leq t'  \leq t$,
	\be\label{e.energy2}
		E(t) + \int_{t'}^t D(s) \ds
			\lesssim E(t) + \int_{t'}^t \int \nabla_v f a \nabla_v f \dz
			\leq E(t'),
	\ee
	where we borrow the notation for $E$ and $D$ from the proof of~\Cref{t.main} (see~\eqref{e.definitionED}).  Let us point out that, as a result of~\eqref{e.energy2}, $E$ is decreasing in time.   
	
	As in the proof of \Cref{t.main}, we note that
	\be\label{e.c062001}
		\int_{t'}^t D(s) \ds \approx \lbr f\rbr_{H^1_\kin([t',t]\times \R^{2d})}^2
	\ee
	(recall~\eqref{e.H1_kin_f}).
	
	Now we introduce the Nash inequality to control the quantities in~\eqref{e.energy2}.  %{e.c062001}.  
	Indeed, applying \Cref{t.ws_Nash} 	with the choices $s_0 = \sqrt{\sfrac{t}{4}}$, $s = \sqrt{\eps t}$ for $\eps\in (0,\sfrac14)$ to be chosen,
	\be
		\Omega_1 = [\sfrac{t}{2},t] \times \R^{2d},
		\quad
		\Omega_2 = [\sfrac{t}{4},t] \times \R^{2d},
		\quad\tand\quad
		B = \{z \in (0,1]\times \R^{2d} : \dist(z,0) \leq 1\},
	\ee
	we deduce that
	\be\label{e.c062005}
		\int_{\sfrac{t}{2}}^t E(s) \ds
			\lesssim \frac{\eps t}{\delta} \int_{\sfrac{t}{4}}^t D(s) \ds
				+ \delta \int_{\sfrac{t}{4}}^t E(s) \ds
				+ \frac{t^2}{(\eps t)^{2d+1}} \Big( \int f_\init \dx\dv\Big)^2,
	\ee
	where $\delta>0$ is a parameter to be chosen and where we applied~\eqref{e.c062004}.  
%	Inserting this into~\eqref{e.energy2} yields	
%	\be
%		E(\sfrac{t}{4}) - E(t)
%			\gtrsim \frac{\delta}{\eps t} \Big[ \Big( \int_{\sfrac{t}{2}}^t E(s) \ds
%				- \delta  \int_{\sfrac{t}{4}}^t E(s) \ds
%				- \frac{t^2}{(\eps t)^{2d+1}} \Big( \int f_\init \dx\dv\Big)^2\Big].
%	\ee

Fix $\overline \alpha>0$ to be chosen and let
\be
	\alpha = \overline \alpha \Big( \int f_\init \dx\dv\Big)^2.
\ee
Define
\be
	t_0 = \sup \Big\{ t : E(s) \leq \frac{\alpha}{s^{2d}}
					\quad \text{ for all } s\in (0,t]\Big\}.
\ee
Up to approximation, we may assume that $f_\init$ is smooth and compactly supported, whence $t_0 > 0$.  Our goal is to show that $t_0 = \infty$.  Hence, we argue by contradiction assuming that $t_0$ is finite.

Clearly $E(t_0) = \alpha t_0^{-2d}$ and $E(s) \leq \alpha s^{-2d}$ for all $s\leq t_0$.  Using this in~\eqref{e.c062005} and recalling that $E$ is decreasing in time yields
\be\label{e.c0603006}
	\frac{\alpha}{t_0^{2d-1}}
			= t_0 E(t_0)
			\lesssim \frac{\eps t_0}{\delta} \int_{\sfrac{t_0}{4}}^{t_0} D(s) \ds
				+ \frac{\delta \alpha}{t_0^{2d-1}}
				+ \frac{\alpha}{ \bar\alpha \eps^{2d+1} (\eps t_0)^{2d-1}}.
\ee
Using~\eqref{e.energy2} and the definition of $t_0$, we have
\be
	\int_{\sfrac{t_0}{4}}^{t_0} D(s) \ds
		\leq E(\sfrac{t_0}{4})
		\leq \frac{4^{2d}\alpha}{t_0^{2d}}.
\ee
Including this in~\eqref{e.c0603006}, we find
\be
	\frac{\alpha}{t_0^{2d-1}}
			\lesssim \frac{\eps }{\delta} \frac{\alpha}{t_0^{2d-1}}
				+ \frac{\delta \alpha}{t_0^{2d-1}}
				+ \frac{\alpha}{ \bar\alpha \eps^{2d+1} (\eps t_0)^{2d-1}}.
\ee
This is clearly a contradiction after choosing, in order, $\delta$ and $\epsilon$ small and $\overline \alpha$ large.  It follows that $t_0 = \infty$.

The rest of the proof is simple functional analysis. 
From the fact that $t_0=\infty$, we deduce that
\be
	\int_{\R^{2d}} f(t,x,v)^2 \dx\dv
		\lesssim \frac{1}{t^{2d}} \int f_\init \dx\dv.
\ee
This can be rephrased as
\be
	\|S_t\|_{L^1 \to L^2} \lesssim \frac{1}{t^d}.
\ee
Of course, the same inequality follows for the solution operator $\tilde S_t$ adjoint equation
\be
	(\partial_t - v\cdot\nabla_x) f = \nabla_v\cdot\left( a \nabla_v f\right)
%	(\partial_t - v\cdot\nabla_x) \tilde f = \Delta_v \tilde f
		\qquad\text{ in } \R_+ \times \R^{2d}.
\ee
Hence, the operator $\tilde S^*_t: L^2 \to L^\infty$, which is again a solution operator of~\eqref{e.ws_kfp}, satisfies
\be
	\|\tilde S^*_t\|_{L^2 \to L^\infty} \lesssim \frac{1}{t^d}.
\ee
Writing $f(t) = \tilde S^*_{\sfrac{t}{2}}\tilde S^*_{\sfrac{t}{2}}f_\init$, we deduce
\be
	\|f(t)\|_{L^\infty(\R^{2d})}
		= \|\tilde S^*_{\sfrac{t}{2}} S_{\sfrac{t}{2}}f_\init\|_{L^\infty(\R^{2d})}
		\lesssim \frac{1}{t^d} \|S_{\sfrac{t}{2}}f_\init\|_{L^2(\R^{2d})}
		\lesssim \frac{1}{t^d}\frac{1}{t^d} \|f_\init\|_{L^1(\R^{2d})},
\ee
which concludes the proof.
\end{proof}

\appendix

\section{A kinetic version of Young's convolution inequality}

Let us note that the main inequality in the following lemma is well-known.  Indeed, it is known as Young's convolution inequality for integrals with respect to the bi-invariant Haar measure associated to a locally compact group.  In our case, the group is $(\R^{2d+1},\circ)$.  We include it here for completeness and, importantly, because the change in domain of integration is crucial to our results above.

\begin{Lemma}\label{l.Youngs}
	Fix any measurable sets $A, B \subset \R^{2d+1}$.  Let $g$ and $\psi$ be measurable functions, and let the indices $r,p,q\in [1,\infty]$ satisfy
	%For any measurable $g$ and $\psi$ and any indices $r,p,q\in [1,\infty]$ satisfying
	\be
		\frac{1}{r} + 1 = \frac{1}{p} + \frac{1}{q}.
	\ee
	If $\supp \psi \subset B$, then 
	\be
	 	\|g*\psi\|_{L^r(A)}
			\leq \|g\|_{L^p(A\circ B^{-1})} \|\psi\|_{L^q(B)}.
	\ee
	We recall the definition~\eqref{e.circ_sets} of $A\circ B^{-1}$ and the definition~\eqref{e.convolution} of the kinetic convolution.
\end{Lemma}
\begin{proof}
	Fix any $z\in A$.  Let us note that 
	\be
		\frac1r + \frac{r-p}{rp} + \frac{r-q}{rq}
			= \frac{1}{p} + \frac{1}{q} - \frac{1}{r}
			= 1.
	\ee
	Hence, applying the generalized H\"older inequality to the suitably re-written convolution, we have 
	\be
	\begin{split}
		|(g*\psi)(z)|
		&\leq 
		\int_B |g(z\circ \tilde z^{-1})||\psi(\tilde z)|\,d\tilde z
		\\&=
		\int_B |g(z\circ \tilde z^{-1})|^{1+\frac{p}{r}-\frac{p}{r}}|\psi(\tilde z)|^{1+\frac{q}{r}-\frac{q}{r}}\,d\tilde z
		\\&=
		\int_B |g(z\circ \tilde z^{-1})|^{\frac{p}{r}}|\psi(\tilde z)|^{\frac{q}{r}}
		|g(z\circ \tilde z^{-1})|^{1-\frac{p}{r}}|\psi(\tilde z)|^{1-\frac{q}{r}}\,d\tilde z
		\\&=
		\int_B \left(|g(z\circ \tilde z^{-1})|^{p}|\psi(\tilde z)|^{q}\right)^{\frac{1}{r}}
		|g(z\circ \tilde z^{-1})|^{\frac{r-p}{r}}|\psi(\tilde z)|^{\frac{r-q}{r}}\,d\tilde z
		\\&\leq
		\left\| \left(|g(z\circ \cdot^{-1})|^{p}|\psi|^{q}\right)^{\frac{1}{r}}\right\|_{L^r(B)}
		\left\| |g(z\circ \cdot^{-1})|^{\frac{r-p}{r}}\right\|_{L^{\frac{pr}{r-p}}(B)}
		\left\| |\psi|^{\frac{r-q}{r}}\right\|_{L^{\frac{qr}{r-q}}(B)}.
	\end{split}
	\ee
	We simplify two of the norms above.  First, we clearly have
	\be
		\left\| |\psi|^{\frac{r-q}{r}}\right\|_{L^{\frac{qr}{r-q}}(B)}
			=
			\left\| \psi\right\|_{L^q(B)}^{\frac{r-q}{r}}.
	\ee
	Additionally, by using that $z\circ B^{-1} \subset A\circ B^{-1}$, we see that
	\be \label{e.c050801}
	\begin{split}
%		&
%		\left\| \left(|g(z\circ \cdot^{-1})|^{p}|\psi|^{q}\right)^{\frac{1}{r}}\right\|_{L^r(B)}
%		=
%		\left(\int_B |g(z\circ \cdot^{-1})|^{p}|\psi|^{q}\,d\tilde z\right)^{1/r},
%	\\
	&
		\left\| |g(z\circ \cdot^{-1})|^{\frac{r-p}{r}}\right\|_{L^{\frac{pr}{r-p}}(B)}
		=
		\left\| g(z\circ \cdot^{-1})\right\|_{L^{p}(B)}^{\frac{r-p}{r}}
		=
		\left\| g(z\circ \cdot)\right\|_{L^{p}(B^{-1})}^{\frac{r-p}{r}}
		\leq
		\|g\|_{L^p(A\circ B^{-1})}^{\frac{r-p}{r}}.
	\end{split}
	\ee	
	Here, we used that the Jacobian associated to $\tilde z \mapsto \tilde z^{-1}$ is one.
	
	In summary, we have arrived at, for any fixed $z\in A$,
	\be
		|(g*\psi)(z)|
		\leq 
		\left(\int_B |g(z\circ \tilde z^{-1})|^{p}|\psi(\tilde z)|^{q}\,d\tilde z\right)^{1/r}
		\|g\|_{L^p(A\circ B^{-1})}^{\frac{r-p}{r}}
		\left\| \psi\right\|_{L^q(B)}^{\frac{r-q}{r}}.
	\ee
	We now integrate over all $z \in A$, use the Fubini-Tonelli theorem, and enlarge the domains as we did in~\eqref{e.c050801} to find:
	\be
		\begin{split}
		\|g*\psi\|_{L^r(A)}^r
%		&=
%		\int_A |(g*\psi)(z)|^r\,dz
%		\\
		&\leq
		\left\| g\right\|_{L^p(A\circ B^{-1})}^{r-p}
		\left\| \psi\right\|_{L^q (B)}^{r-q}
		\int_A \left(\int_B |g(z\circ \tilde z^{-1})|^{p}|\psi(\tilde z)|^{q}\,d\tilde z\right)
		\dz
		\\&\leq
		\left\| g\right\|_{L^p(A\circ B^{-1})}^{r-p}
		\left\| \psi\right\|_{L^q (B)}^{r-q}
		\int_B
		|\psi(\tilde z)|^{q} \left(\int_A |g(z\circ \tilde z^{-1})|^{p}\dz \right)\dtz
		\\&\leq
		\left\| g\right\|_{L^p(A\circ B^{-1})}^{r-p}
		\left\| \psi\right\|_{L^q (B)}^{r-q}
		\int_B
		|\psi(\tilde z)|^{q} \left(\int_{A\circ B^{-1}} |g(\zeta)|^{p}\,d\zeta \right)\dtz
		\\&=
		\left\| g\right\|_{L^p(A\circ B^{-1})}^r
		\left\| \psi\right\|_{L^q (B)}^r.
		\end{split}
	\ee
	The proof is complete after taking each side to the $\sfrac1r$ power.
\end{proof}

%
%
%
%
%
%
%
%
%\begin{thebibliography}{9}
%%
%%\bibitem{PP}
%%  A. Pascucci, A. Pesce,
%%  \textit{Sobolev embeddings for kinetic Fokker-Planck equations},
%%Preprint, arXiv:2209.05124, (2022).
%%
%%\bibitem{tartar}
%%Tartar, Luc,
%%     \textit{An introduction to {S}obolev spaces and interpolation spaces},
%%     Lecture Notes of the Unione Matematica Italiana,
%%     3,
%%     Springer, Berlin; UMI, Bologna,
%%     2007,
%%     xxvi+218
%%     Joan L. Cerd\`a.
%%
%%
%\bibitem{integrated}
%  Groeneboom, Piet and Jongbloed, Geurt and Wellner, Jon A.,
%  \textit{Integrated {B}rownian motion, conditioned to be positive},
%Ann. Probab., arXiv:2209.05124, (1999).
%
%\bibitem{special}
% Olver, Frank W. J.,
%  \textit{Asymptotics and special functions},
%AKP Classics, Reprint of the 1974 original [Academic Press, New York (1997).
%
%
%%@book {MR1429619,
%%    AUTHOR = {},
%%     TITLE = {Asymptotics and special functions},
%%    SERIES = {AKP Classics},
%%      NOTE = {Reprint of the 1974 original [Academic Press, New York;
%%              MR0435697 (55 \#8655)]},
%% PUBLISHER = {A K Peters, Ltd., Wellesley, MA},
%%      YEAR = {1997},
%%     PAGES = {xviii+572},
%%      ISBN = {1-56881-069-5},
%%   MRCLASS = {41-02 (33Cxx 41A60 65D20)},
%%  MRNUMBER = {1429619},
%%}OLVER, F. W. J. Ž1974.. . Academic Press, New York.
%\end{thebibliography}
\bibliographystyle{abbrv}
\bibliography{kinetic_nash}

\end{document}